\documentclass[]{amsart}
\usepackage{amscd,amsthm,amssymb,amsfonts,amsmath,euscript}

\theoremstyle{plain}
\newtheorem{thm}{Theorem}[section]
\newtheorem{lem}[thm]{Lemma}
\newtheorem{prop}[thm]{Proposition}

\theoremstyle{definition}
\newtheorem{defn}[thm]{Definition}

\theoremstyle{remark}
\newtheorem{rem}[thm]{Remark}
\newtheorem*{Rem}{Remark}

\newtheorem*{thank}{{\bf Acknowledgments}}
\newcommand{\nc}{\newcommand}

\def\makeop#1{\expandafter\def\csname#1\endcsname
  {\mathop{\rm #1}\nolimits}\ignorespaces}
\makeop{Hom}   \makeop{End}   \makeop{Aut}   \makeop{Isom}  \makeop{Pic}
\makeop{Gal}   \makeop{ord}   \makeop{Char}  \makeop{Div}   \makeop{Lie}
\makeop{PGL}   \makeop{Corr}  \makeop{PSL}   \makeop{sgn}   \makeop{Spf}
\makeop{Spec}  \makeop{Tr}    \makeop{Nr}    \makeop{Fr}    \makeop{disc}
\makeop{Proj}  \makeop{supp}  \makeop{ker}   \makeop{im}    \makeop{dom}
\makeop{coker} \makeop{Stab}  \makeop{SO}    \makeop{SL}    \makeop{SL}
\makeop{Cl}    \makeop{cond}  \makeop{Br}    \makeop{inv}   \makeop{rank}
\makeop{id}    \makeop{Fil}   \makeop{Frac}  \makeop{GL}    \makeop{SU}
\makeop{Nrd}   \makeop{Sp}    \makeop{Tr}    \makeop{Trd}   \makeop{diag}
\makeop{Res}   \makeop{ind}   \makeop{depth} \makeop{Tr}    \makeop{st}
\makeop{Ad}    \makeop{Int}   \makeop{tr}    \makeop{Sym}   \makeop{can}
\makeop{length}\makeop{SO}    \makeop{torsion} \makeop{GSp} \makeop{Ker}
\makeop{Adm}
\def\makebb#1{\expandafter\def
  \csname bb#1\endcsname{{\mathbb{#1}}}\ignorespaces}
\def\makebf#1{\expandafter\def\csname bf#1\endcsname{{\bf
      #1}}\ignorespaces}
\def\makegr#1{\expandafter\def
  \csname gr#1\endcsname{{\mathfrak{#1}}}\ignorespaces}
\def\makescr#1{\expandafter\def
  \csname scr#1\endcsname{{\EuScript{#1}}}\ignorespaces}
\def\makecal#1{\expandafter\def\csname cal#1\endcsname{{\mathcal
      #1}}\ignorespaces}

\def\doLetters#1{#1A #1B #1C #1D #1E #1F #1G #1H #1I #1J #1K #1L #1M
                 #1N #1O #1P #1Q #1R #1S #1T #1U #1V #1W #1X #1Y #1Z}
\def\doletters#1{#1a #1b #1c #1d #1e #1f #1g #1h #1i #1j #1k #1l #1m
                 #1n #1o #1p #1q #1r #1s #1t #1u #1v #1w #1x #1y #1z}
\doLetters\makebb   \doLetters\makecal  \doLetters\makebf
\doLetters\makescr
\doletters\makebf   \doLetters\makegr   \doletters\makegr
     \def\qed{\qedmark\medbreak}%
\def\qedmark{{\enspace\vrule height 6pt width 5pt depth 1.5pt}}%

\normalsize

\makeop{Bl}

\def\Fq{{\bbF}_q}

\newcommand{\Z}{\mathbb Z}

\newcommand{\F}{\mathbb F}

\newcommand{\isoto}{\stackrel{\sim}{\longrightarrow}}
\nc{\embed}{\hookrightarrow}

\nc{\ol}{\overline}
\nc{\wt}{\widetilde}
\nc{\opp}{\mathrm{opp}}
\def\ul{\underline}

\def\Mat{{\rm Mat}}

\makeop{Ram}
\makeop{Rep}

\numberwithin{equation}{section}

\begin{document}

\title[Class numbers of definite central simple algebras]{Class numbers of
  central simple algebras over global function fields}
\author[Wei]{Fu-Tsun Wei}
\address{(Wei) Department of Mathematics, National Tsing-Hua
  University, Hsinchu 30013, Taiwan}
\email{ftwei@mx.nthu.edu.tw}
\author[Yu]{Chia-Fu Yu}
\address{
(Yu) Institute of Mathematics, Academia Sinica and NCTS (Taipei
Office) \\
6th Floor, Astronomy Mathematics Building\\ 
No. 1, Roosevelt Rd. Sec. 4 \\
Taipei 10617, Taiwan}
\email{chiafu@math.sinica.edu.tw}

\date{\today}

\begin{abstract}
Let $K$ be a global function field
together with a place $\infty$, and $A$ the subring of functions
regular outside $\infty$. In this paper we present an effective method to evaluate the (locally free)
class number of an arbitrary hereditary $A$-order 
in an arbitrary definite central
simple $K$-algebra.   
We also show that the class number of 
any non-principal genus for a hereditary order in $D$ 
can be reduced to that of the principal genus for another hereditary 
order in $D$.
\end{abstract}

\maketitle

\section{Introduction}\label{sec:01}
Let $K$ be a global function field with finite constant field $\Fq$
together 
with a place $\infty$ (the place at infinity), 
and $A$ the subring of functions regular
outside $\infty$.
Let $D$ be a definite (with respect to $\infty$) central simple
algebra over $K$ of degree $n$, that is,
$D\otimes_K K_{\infty}$ is a division algebra with $\dim_K D = n^2$,
where $K_{\infty}$ is the completion of $K$ at the place $\infty$.
Let $R$ be a hereditary $A$-order in $D$. Recall that an $A$-order in
$D$ is said to be {\it hereditary} if any
left (or right) ideal of $R$ is a projective $R$-module. 
Denote by $h(R)$ the class number of $R$.  
By this we mean the number of equivalence classes of locally free right
(or left) ideals of $R$. Note that for the non-commutative rings, projective
modules may not be locally free.  
The computation of the class number $h(R)$ has been done 
in the following cases:

\begin{enumerate}
\item The algebra $D$ is quaternion, due to Eichler \cite{Eichler}.
\item The algebra $D$ is of Drinfeld type, 
  $A=\Fq[t]$ is
  a polynomial 
  ring and $R$ is a maximal $A$-order, due to Gekeler
  \cite{Ge3}.
\item The algebra $D$ is of prime index and $A=\Fq[t]$ is a polynomial
  ring, due to Denert and Van Geel \cite{D-G}.
\end{enumerate}

Recall that a central simple algebra $D$ over $K$ of degree $n$
is called {\it of Drinfeld type} if  $D$ is ramified
  precisely at
  the place $\infty$ and another finite place $v$
  with local invariants $-1/n$ and
  $1/n$, respectively. These precisely appear as endomorphism algebras
  of supersingular Drinfeld $A$-modules of rank $n$
  over an algebraic closure $\ol {\bbF}_v$ of the residue field
  $\bbF_v$ of $v$. 

The proof of Gekeler's class number formula is
  geometric. When $R$ is a maximal order, it is well-known
  (see \cite{drinfeld:1}, also cf.\ \cite{Ge2}, \cite{Y-Y})
  that there is a natural bijection between the set $\Cl (R)$ of
  equivalence classes of locally free right ideals of $R$ and the set
  $\Lambda(n,v)$ of
isomorphism
classes of supersingular Drinfeld $A$-modules of rank $n$
over $\overline{\mathbb{F}}_v$. Using this interpretation, Gekeler
calculated the number of isomorphism classes of supersingular Drinfeld $A$-modules in the
fine Drinfeld moduli spaces when $A=\Fq[t]$. Then he
established the {\it transfer principle} which relates the
supersingular Drinfeld $A$-modules with extra symmetries and
supersingular Drinfeld $A'$-modules
of low rank for certain integral extension $A'$ of $A$; see \cite{Ge3} 
for more details. This gives a way to express the class number $h(R)$
recursively.  In \cite[p.~333]{Ge3} Gekeler asked whether
or not the transfer principle holds for a larger class of definite
central simple algebras $D$ than those of Drinfeld type. 
We examine this in the case
where the degree $n$ of $D$ is a prime number. It turns out that the
(naive) transfer principle holds only for a very restricted class:
only for $D$ with exactly two ramified places,
$\infty$ and another finite place $v$ (but allowing different local
invariants at $\infty$ and the finite ramified place $v$).   


The proof of the class number formula for definite 
central division algebras of
prime degree by Denert and Van Geel \cite{D-G} 
is based on the generalization of Eichler's trace formula for 
Brandt matrices (see \cite{Eichler} and \cite[Formula (7), p. 392]{D-G}), 
and to compute explicitly the terms in 
the Eichler-Brandt trace formula in the case where 
$K$ is the rational function field.  
As far as the authors know, the trace formula of Brandt matrices 
is known only for definite central division algebras of prime degree 
(see \cite[Theorem 4.3]{brzezinski:auto_crelle1990} and 
\cite[Formula (7), p. 392]{D-G}),
and the explicit computation of the terms in the 
Eichler-Brandt trace formula is carried out by Denert and Van Geel 
only when $K$ is the rational function field 
(see \cite[Theorems 3 and 9]{D-G}).    

In this article we evaluate the class
number $h(R)$ for the general case, namely, for any global
function field $K$ with a place at infinity, any definite central
simple algebra $D$ and any hereditary
$A$-order $R$ in $D$. Our approach is to construct an algebraic
analogue of Gekeler's transfer principle for arbitrary definite
central simple algebras, 
which shares the same spirit of relating class numbers of algebras 
of larger degrees to those of lower degrees, but allowing 
more complicated 
relations among them than the original one 
(as we know, the naive generalization of Gekeler's transfer
principle only holds in a very restricted class). 
We now describe this transfer principle.

                 
First observe that the multiplicative group $R'^\times$ 
of any $A$-order $R'$ in $D$ is a finite cyclic group. More precisely, 
$R'^\times\simeq \F_{q^s}^\times$ for a finite field $\F_{q^s}$
contained in $D$, and hence $s|n$. 
Let $I_1,...,I_{h(R)}$ be representatives of locally free right ideal classes of
$R$, and let $R_i$ be the left order of $I_i$ for $1\leq i \leq
h(R)$. For computing the class number $h(R)$, we may
compute what we call the {\it weight-$s$ class number} 
$$h_s(D/K,R) = \#\{ i \mid 1\leq i \leq h(R), R_i^{\times} \cong
\mathbb{F}_{q^s}^{\times}\}$$ 
for every positive divisor $s$ of $n$. 
A basic result says that an $A$-order $R$ is hereditary if and only if
all its local completions $R_v$ are hereditary. 
Also, the hereditary order $R_v$ is determined 
by its {\it invariant} $\vec{f}_v=(f_{v,1},\dots, f_{v, r_v})$ up to
conjugation;  
see Section~\ref{sec:22}.  As a result, 
the class number $h_s(D/K,R)$ depends only 
on the {\it invariant} 
$\vec{\mathbf{f}} := (\vec{f}_v)_{v \neq
\infty}$ of 
$R$, so we also denote it by $h_s(D/K,\vec{\mathbf{f}})$.

Let $\F_D$ be a maximal finite subfield of $D$; the degree
$s_0:=[\F_D:\F_q]$ does not depend on the choice of $\F_D$, following
from the Noether-Skolem Theorem.
From the basic properties, one knows 
\[ h_s(D/K,R) = 0\quad \text{ if }\ s \nmid s_0. \]
Let $s$ be a positive divisor of $s_0$. Let 
$L_s:= K\cdot \F_{q^s}$ be the constant field extension of $K$ of
degree $s$, and let  $O_{L_s}$ denote the integral closure of $A$ in
$L_s$.  
Choose an embedding  $\iota: L_s\to D$ of $L_s$ into $D$,
and let $D'_s$ be the centralizer of its image $\iota(L_s)$ in $D$. 
We regard $L_s$ as a subfield of $D$ via the embedding $\iota$. Our 
main results (see Theorem~\ref{11}) do not depend on the choice of
the embedding $\iota$. 
It is easy to see that $D'_s$ is again a definite central simple algebra
over $L_s$ with respect to the unique place $\infty_s$ over
$\infty$, with degree $n/s$ and $[\F_{D_s'}: \F_{q^s}] = [\F_D:
\F_q]/s$.

Next we define a finite index set
$\Omega(D/K,s,\vec{\mathbf{f}})$. 
This set is actually the combinatorial
description of the product of all local optimal
embeddings, where we give a detailed analysis in
Section~\ref{sec:07}, which is the core of our explicit computation. 
Let $\Sigma_K^0$ (resp.~$\Sigma_{L_s}^0$) the set of all finite places
of $K$ (resp.~of $L_s$).
Let $v$ be a finite place of $K$. Let $D_v:=D\otimes K_v\simeq
\Mat_{m_v}(\Delta_v)$, where $\Delta_v$ is a central division algebra
over $K_v$ and $m_v$ is the {\it local capacity} at $v$ (see \cite{Rei}). 
Let $d_v$ be the degree of $\Delta_v$; one has $n=m_v
d_v$ for all $v$ and $d_v=1$ for almost all places $v$. Note that if 
$\vec{f}_v=(f_{v,1},\dots, f_{v, r_v})$ is the invariant of the local
hereditary order $R_v$, then $\sum_{i=1}^{r_v} f_{v,i}=m_v$.  
Put 
\[ \ell_{s,v} := \text{gcd}(s, \deg v),\quad  \text{and} \quad
t_{s,v} := \text{gcd}(s/\ell_{s,v}, d_v). \] 




Let $\Omega_v(D/K,s,\vec{f}_v)$ denote 
the set consisting of all tuples $(\vec{f}_{w,*})_{w|v}$ indexed by
places $w$ of $L_s$ over $v$, 
where each $\vec{f}_{w,*}=(f_{w,(i,j)})$ is an $r_v\times
t_{s,v}$-matrix with  
non-negative integer entries $f_{w,(i,j)}\in \Z_{\ge 0}$ 
($1\le i\le r_v$ and $1\le j \le t_{s,v}$), that satisfy
the following conditions:  
If one puts 
\begin{equation}
  \label{eq:11}
f_{w,i}:=\frac{s}{\ell_{s,v}
    t_{s,v}}\cdot \sum\limits_{j=1}\limits^{t_{s,v}} f_{w,(i,j)},  
\end{equation}
then  
\begin{equation}
  \label{eq:12}
  \sum\limits_{i=1}\limits^{r_v} f_{w,i} = \frac{m_v}{\ell_{s,v}}
\quad \forall\, w \mid v, \text{ and } 
\sum\limits_{w\mid v} f_{w,i} = f_{v,i}\quad 
\text{for } 1\leq i \leq r_v.
\end{equation}
Then we define the index set $\Omega (D/K, s,\vec{\mathbf{f}})$ by
\begin{equation}
  \label{eq:13}
  \Omega (D/K, s,\vec{\mathbf{f}}):=\prod_{v\in \Sigma_K^0}
  \Omega_v(D/K,s,\vec{f}_v). 
\end{equation}
It is not hard to see that the 
local component $\Omega_v(D/K,s,\vec{f}_v)$ is
singleton for almost all $v$ and every $\Omega_v(D/K,s,\vec{f}_v)$ is a
finite set, and hence that 
$\Omega (D/K, s,\vec{\mathbf{f}})$ is a finite set. 

We can write any element in $\Omega (D/K, s,\vec{\mathbf{f}})$ in the
form  
$\vec{\mathbf{f}}_* = (\vec{f}_{w,*})_{w
  \in \Sigma_{L_s}^0}$, 
where $\vec{f}_{w,*}=(f_{w,(i,j)})$ is an element in 
$\Z_{\ge 0}^{r_v}\times \Z_{\ge 0}^{t_{s,v}}$
with the conditions above.
We make an
appropriate order on the index set $\{(i,j)\}$ 
and regard $\vec{f}_{w,*}$ as a (long) vector 
in $\Z^{r_v\cdot t_{s,v}}_{\ge 0}$; see Section~\ref{sec:42} for
details. 
Denote by $\vec{f}_{w,*}^o$ the vector obtained by removing zero
entries of the vector $\vec{f}_{w,*}$ (also see Section~\ref{sec:42}) 
and define 
$\vec{\mathbf{f}}_*^o:= (\vec{f}_{w,*}^o)_{w \in \Sigma_{L_s}^0}$.
Note that the sum $\sum_{i,j} f_{w,(i,j)}$ is equal to the local
capacity of the central simple algebra $D'_s$ over $L_s$ at $w$ for all
finite places $w$. Therefore, there is a hereditary order 
$R'(\vec{\mathbf{f}}_*^o)$ in $D'_s$ with 
invariant $\vec{\mathbf{f}}_*^o$ and 
it makes sense to talk about the class
numbers  $h_{s''}(D_s'/L_s, \vec{\mathbf{f}}_*^o):=h_{s''}(D_s'/L_s,
R'(\vec{\mathbf{f}}_*^o))$ for positive divisors $s''$ of $[\F_{D_s'}: \F_{q^s}]$.
    

With notations being as above, our transfer principle is stated as
follows. 

\begin{thm}\label{11} 
For any two 
positive divisors $s$ and $s'$ of $[\F_D:\F_q]$ with 
$s \mid s'$, we have 
\begin{equation}
  \label{eq:14}
  s \cdot h_{s'}(D/K,R) = \sum_{\vec{\mathbf{f}}_* \in
  \Omega(D/K,s,\vec{\mathbf{f}})} h_{s'/s}(D_s'/L_s, 
R'(\vec{\mathbf{f}}_*^o)).
\end{equation}
\end{thm}



We make a few remarks about Theorem~\ref{11}.
When $D$ is of Drinfeld type, Theorem~\ref{11} recovers Gekeler's
transfer principle (cf. Theorem~\ref{51}). 
See Section~\ref{sec:51} for the explanation of Gekeler's transfer
principle for supersingular Drinfeld modules and 
the deduction from Theorem~\ref{11}. As a result, we give an
algebraic proof of Theorem~\ref{51}. 
Note that in this special case, the hereditary 
orders $R'(\vec{\mathbf{f}}_*^o)$ 
in $D'_s$ occurred in (\ref{eq:14}) are also maximal. This is not true 
for the general cases; one needs to deal with class numbers of
hereditary orders (in smaller subalgebras) as well 
even when one starts with a maximal order $R$ in $D$. 
This is also a main reason for us to consider directly the cases of arbitrary hereditary orders.     



Theorem~\ref{11} indicates that in order to get $h_s(D/K,R)$, it
suffices to compute $h_1(D_s'/L_s,R'(\vec{\mathbf{f}}_*^o))$ for each
$\vec{\mathbf{f}}_*$ in $\Omega(D/K,s,\vec{\mathbf{f}})$. 
Recall that the associated mass $\text{Mass}(D/K,R)$ is defined as
follows: 
$$\text{Mass}(D/K,R) := \sum_{1\leq i \leq h(R)}
\frac{1}{\#(R_i^{\times})} 
= \sum_{s \mid [\F_D:\F_q]} \frac{h_s(D/K,R)}{q^s-1},$$
where $R_i$ is the left order of each $R$-ideal $I_i$. The mass 
$\text{Mass}(D/K,R)$ only depends on the invariant $\vec{\bff}$ of
$R$, so we may also denote it by $\text{Mass}(D/K,\vec{\bff})$.
Notice that $\#(R_i^{\times}) = q^s-1$ if there exists an optimal
embedding of  $O_{L_s}$ into $R_i$. 
The mass formula (cf.\ \cite{D-G}, the precise formula is also stated
in Theorem~\ref{31}) shows that the mass $\text{Mass}(D/K,R)$ can be
computed 
explicitly in terms of integral values of the zeta function of $K$. 
Theorem~\ref{11} and the mass formula for 
$\text{Mass}(D_s'/L_s,R'(\vec{\mathbf{f}}_*^o))$ together provide enough
equations to solve each weight-$s$ class number $h_s(D/K,R)$.
As a result, the class number $h(R)$ can be also
expressed in terms of special zeta values eventually. 

The second part of main results is to make the computation of the class
number more effectively. Note that the index set $ \Omega (D/K,
s,\vec{\mathbf{f}})$ in (\ref{eq:14}) 
is the product of its local components $\Omega_v$
(\ref{eq:13}). However, the class number $h_{s'/s}(D_s'/L_s, 
R'(\vec{\mathbf{f}}_*^o))$ is not.
Therefore, there is no direct way to reduce the 
computations locally.  
The way we do is to regroup the relation (\ref{eq:14}) 
into the (partial) mass sums (see Theorem~\ref{45}):
\begin{equation}
  \label{eq:15}
  s \cdot \sum_{s': s \mid s' \mid [\F_D:\F_K]}
\frac{h_{s'}(D/K,\vec{\mathbf{f}})}{q^{s'}-1} 
 = \sum_{\vec{\mathbf{f}}_* \in \Omega(D/K,s,\vec{\mathbf{f}})}
\text{\rm Mass}(D_s'/L_s,\vec{\mathbf{f}}_*^o).
\end{equation}
Note that now one can compute the right side of the equation (\ref{eq:15})
directly using the mass formula, 
without going through the induction step of computing
$h_1(D_s'/L_s,R'(\vec{\mathbf{f}}_*^o))$ recursively. 
This simplifies the computation significantly. 

The second input is using the nice property
of the masses, which allows us to 
separate the global and local contributions.
Set 
\[ \text{Mass}(D'_s/L_s):=\text{Mass}(D'_s/L_s,
\vec{\mathbf{f}}'_{\text{max}}),\]
where
$\vec{\mathbf{f}}'_{\text{max}}$ is the invariant of a maximal
$O_{L_s}$-order $R'_{\text{max}}$ in $D'_s$. For any finite place $w$ of
$L_s$ and any vector $\vec{h}_w = (h_{w,1},...,h_{w,r_w}) \in
\mathbb{Z}_{\geq 0}^{r_w}$, we set 
\begin{equation}
  \label{eq:16}
  \mathcal{T'}_w(D'_s/L_s,\vec{h}_w):=
\frac{\prod\limits_{i=1}\limits^{m_w} \big( N(w)^{d_w
    i}-1\big)}{\prod\limits_{i=1}\limits^{r_w}
  \left(\prod\limits_{j=1}\limits^{h_{v,i}}\big(N(w)^{d_w
      j}-1\big)\right)},
\end{equation}
where $m_w$ is the local capacity of $D'_s$ at $w$, $d_w$ is the local
index of $D'_s$ at $w$ and $N(w)$ is the cardinality of the residue
field of $L_s$ at $w$. 
Then the mass formula (Theorem~\ref{31}) gives   
$$\text{Mass}(D'_s/L_s,\vec{\mathbf{f}}_*^o)) = \text{Mass}(D'_s/L_s) 
  \cdot \prod_{w  \in \Sigma_{L_s}^0} 
\mathcal{T'}_w(D'_s/L_s,\vec{f}_{w,*}).$$
 
For each finite place $v$ of $K$, define
\begin{equation}
  \label{eq:17}
  \Theta_v (D/K,s,\vec{f}_v):= \sum_{(\vec{f}_{w,*})_{w \mid v} \in
  \Omega_v(D/K,s,\vec{f}_v)} \left(\prod_{w \mid v}
  \mathcal{T'}_w(D_s'/L_s,\vec{f}_{w,*})\right).
\end{equation}

Our second main result (Theorem~\ref{46}) 
states as follows, which reduces the
computation in purely local terms. 

\begin{thm}\label{12}
Notations being as above, one has
\begin{equation}
  \label{eq:18}
  s \cdot \sum_{s': s \mid s' \mid [\F_D:\F_K]}
\frac{h_{s'}(D/K,\vec{\mathbf{f}})}{q^{s'}-1} 
= \text{\rm Mass}(D_s'/L_s) \cdot \prod_{v \in \Sigma_K^0}
\Theta_v(D/K,s,\vec{f}_v).
\end{equation}
\end{thm}

We also provide a simple method to compute the local terms 
$\Theta_v(D/K,s,\vec{f}_v)$ using generating functions (see
Proposition~\ref{48}). In Section~\ref{sec4.1} we present a recursive
formula for computing the class number $h(R)$ using Theorems~\ref{11}
and \ref{12} together with the explicit computation of the local terms
$\Theta_v(D/K,s,\vec{f}_v)$. 

We already mentioned that the index set 
$\Omega(D/K,s,\vec{\mathbf{f}})$ is the combinatorial
description of the product of  all local optimal
embeddings of $L_s$ in $R_i$, following from the results in
Section~\ref{sec:07}. 
Therefore, this gives a combinatorial criterion of 
determining the existence of optimal embeddings. 
As a by-product, we obtain the
following generalization (Theorem~\ref{435}) 
of Eichler's theorem \cite{Eichler} 
on optimal embeddings from quaternion algebras to central simple
algebras, which is of interest in its own right. 

\begin{thm}\label{13}
  Notations being as above. There is an optimal embedding of
  $O_{L_s}$ into a hereditary order $R'$ in $D$ of invariant
  $\vec{\bff}$ if and only if for all $v\in \Sigma_K^0$, one has the
  divisibility 
\begin{equation}
  \label{eq:19}
  \frac{s}{\ell_{s,v} t_{s,v}}\, {\Big |}\, f_{v,i}
\end{equation}
for $1\leq i \leq r_v$.
\end{thm}   


In \cite{Pa1} Papikian establishes a bijection between the
isomorphism classes of exceptional $\mathcal{D}$-elliptic sheaves with
the ideal classes of a certain hereditary order. Also, he
identified the isomorphism classes of supersingular
$\mathcal{D}$-elliptic sheaves with the ideal classes of a
maximal order.
One application of our main results evaluates the number of 
isomorphism classes of those $\mathcal{D}$-elliptic sheaves.

Another application of our results is to 
compute the dimension of certain automorphic
forms for definite central simple algebras. Indeed, the class number
$h(R)$ 
is simply the dimension of the space $L^2(D^\times \backslash
\widehat{D}^\times / \widehat{R}^\times)$ of automorphic forms. Here
$\widehat{D}:= D\otimes_K \mathbb{A}_K^{\infty}$ and $\widehat{R}:= R
\otimes_A \widehat{A}$, where $\mathbb{A}_K^{\infty}$ is the ring of
finite adeles of $K$ and $\widehat{A}$ is the pro-finite completion of
$A$. \\ 



This paper is organized as follows.
In Section \ref{sec:02}, we include the preliminaries for central simple
algebras over global function fields and local properties of
hereditary orders. We also compute the degree $[\F_D:\F_q]$ of a
maximal finite subfield $\F_D$ in $D$. 
In Section~\ref{sec:03} we explain the strategy of
computing the class number from the mass formula and optimal
embeddings. 
An explicit mass formula for hereditary $A$-orders is described in
Section~\ref{sec:31}. In Section~\ref{sec:32} 
we translate the problem of computing 
class numbers into that for numbers 
of (global) optimal embeddings.
We also note at the end of Section \ref{sec:03} that global
optimal embeddings can be understood through the study of local optimal
embeddings. The main theorems are presented in Section~\ref{sec:04},
and we give 
a recursive formula for computing the class number in
Section~\ref{sec4.1}. 
Section~\ref{sec:05} includes special cases and examples to show the
relation between the main theorems and previous results in \cite{D-G}
and \cite{Ge2}. 
In Section~\ref{sec:51}, we recall Gekeler's transfer principle for
supersingular Drinfeld modules and deduce Gekeler's result from
Theorem~\ref{11}.  
In Section~\ref{sec5.2}, we focus on the special case where $[\F_D:\F_q]$ is a
prime, and express the class number $h(R)$ in terms of masses. 
Moreover, when the degree $n$ of $D$ over $K$ is a prime number, 
we obtain an explicit class number formula 
which coincides with the formulas in \cite[Theorems 3 and 9]{D-G} in
the case where $K$ is the rational function field. 
We also give an example for the
reader's interest using the recursive formula of Section~\ref{sec4.1}.
In Section~\ref{sec:06} we show that  
the computation of the class number of 
other genus of (not necessarily locally free) 
$R$-ideals can be reduced to that of {\it locally free}
$R'$-ideals for another hereditary $A$-order $R'$, the principal genus
class number. 
Section~\ref{sec:07} is an independent section, in which we make a
detailed study of local optimal embeddings. This also plays the key
ingredient in the proof of our main theorem Theorem~\ref{11}.
 

\section{Preliminaries}\label{sec:02}

In this section, we set up general notations and establish basic
results of definite central simple algebras
of arbitrary degree which we shall need in this article.
Further details are referred to \cite{Rei}.

\subsection{General settings}\label{sec:21}

Let $K$ be a global function field with finite constant field
$\mathbb{F}_q$, i.e.\
the transcendental degree of $K$ over $\mathbb{F}_q$ is one and
$\mathbb{F}_q$ is algebraically closed in $K$.
For each place $v$ of $K$, the completion of $K$ at $v$ is denoted by
$K_v$,
and we set $O_v$ to be the valuation ring of $K_v$.
Fix a uniformizer $\pi_v$ in $O_v$.
The residue field $O_v/(\pi_v)$ is denoted by $\mathbb{F}_v$, and
$\deg v$ is the degree of $\mathbb{F}_v$ over $\mathbb{F}_{q}$.
There is a canonical embedding $\mathbb{F}_v \hookrightarrow O_v$, and
$O_v$ is in fact isomorphic to the power series ring 
$\mathbb{F}_v[[\pi_v ]]$.
The cardinality of $\mathbb{F}_v$ is denoted by $N(v)$. 

Let $D$ be a central simple algebra over $K$ with 
$\text{dim}_K D = n^2$. 
For each place $v$ of $K$,
we have that $D_v:= D \otimes_K K_v$ is isomorphic to
$\text{Mat}_{m_v}(\Delta_v)$, where $\Delta_v$ is a central division
algebra over $K_v$ with  $\text{dim}_{K_v} \Delta_v = d_v^2$ and $m_v
d_v = n$.
We recall the definition of local invariants of $D$.
Note that $\Delta_v$ contains an unramified
maximal subfield $F_v$ 
(so $[F_v:K_v]=d_v$),
and there exists an element $u_v$ in $\Delta_v$ such that 
$u_v^{d_v} = \pi_v^{\kappa_v}$ where $\kappa_v\in \Z$ with
$\text{gcd}(\kappa_v, d_v) = 1$, and
$$ u_v \alpha_v = \text{Fr}_v(\alpha_v) u_v \text{ for all } \alpha_v
\text{ in } F_v.$$
Here $\text{Fr}_v$ is the {\it Frobenius automorphism of $F_v$ over
$K_v$}, i.e.\
for any $\alpha_v$ in the valuation ring $O_{F_v}$ of $F_v$ we have
$\text{Fr}_v(\alpha_v) \equiv \alpha_v^{N(v)} \pmod {\pi_v O_{F_v}}$.
The {\it local invariant $\text{\rm inv}_v(D)$ of $D$ at $v$} is
defined as
$$ \inv_v(D):=\kappa_v/d_v \text{ } \bmod  \mathbb{Z}
\in \mathbb{Q}/\mathbb{Z},$$
which is independent of the choices of $F_v$, $\pi_v$, and $u_v$.
We call $D$ \it ramified at $v$ \rm if
$\text{inv}_v(D) \not\equiv 0 \text{ } \pmod \Z$, i.e.\
$\Delta_v$ is not equal to $K_v$.
Let $S=S_D$ be the set consisting of places of $K$ where $D$ is
ramified.
It is well-known that $S$ is finite, and that
$$ \sum_{v\in V^K} \text{inv}_v(D) \equiv 0 \text{ }
\pmod \Z,$$
where $\Sigma_K$ denotes the set of all places of $K$.

\subsection{Hereditary orders}\label{sec:22}

Fix a place $\infty$ of $K$, referred as the place at infinity;
and others are referred as finite places of $K$.
Let $A$ be the ring of functions in $K$ regular outside $\infty$.
Recall that a \it hereditary $A$-order $R$ in
$D$ \rm is an $A$-order in $D$ such
that every left (or equivalently right) ideal of $R$ is projective 
as an $R$-module. 
It is known that an $A$-order $R$ in $D$ is hereditary if
and only if its completion $R_v:=R\otimes_A O_v$ is hereditary for all
finite places $v$ of $K$.  

Let $R$ be a hereditary $A$-order in $D$.
For each finite place $v$ of $K$. The completion
$R_v := R \otimes_A O_v$
is a hereditary $O_v$-order in $D_v \cong \text{Mat}_{m_v}(\Delta_v)$.
The unique maximal $O_v$-order in $\Delta_v$ is denoted by
$O_{\Delta_v}$ and 
we set $\mathfrak{P}_v$ to be its maximal (two-sided) ideal.
It is known that there exists a vector $\vec{f}_v =
(f_{v,1},...,f_{v,r_v})$, where $f_{v,1},...,f_{v,r_v}$ are positive
integers such that $\sum_{i=1}^{r_v}f_{v,i} = m_v$,
and $R_v$ is isomorphic to the ring
$\text{Mat}_{m_v}(\vec{f}_v,O_{\Delta_v})$
consisting of elements $X=(X_{i,j})_{1\leq i, j \leq r_v}$ in
$\text{Mat}_{m_v}(O_{\Delta_v})$ such that
$$X_{i,j} \in \begin{cases}
\text{Mat}_{f_{v,i} \times f_{v,j}} (O_{\Delta_v}) & \text{ if $i \leq
  j$,} \\
\text{Mat}_{f_{v,i} \times f_{v,j}} (\mathfrak{P}_v) & \text{ if $i >
  j$.} \end{cases}$$
The number $r_v$ is called the {\it period} of $R_v$;  
the vector $\vec{f}_v:= (f_{v,1},...,f_{v,r_v})$ is called the \it
invariant of $R$ at $v$, \rm which is uniquely determined by $R_v$ 
up to cyclic permutations. When $\bff_v=(1,\dots,1)$, the order $R_v$
is the Iwahori order which is the preimage of the set of upper
triangular 
matrices over $O_{\Delta_v}/\grP_v$.
We shall also call the collection $\bff=(\bff_v)_{v\neq \infty}$ the
{\it invariant} of $R$. 

Note that the class number $h(R)$ only depends on the invariant $\bff$
of $R$ but not on $R$ itself. 
We also write $h(D/K, \bff)$ for the class number number $h(R)$.



\subsection{The constant field of $D$}
\label{sec:23}

Let $K$, $\infty$, $A$ and $D$ be as above.
Assume that $D$ is definite with respect to
$\infty$, that is, the completion $D_\infty:=D\otimes_K K_\infty$ at
$\infty$ is a division algebra.
For any element $\alpha$ in $D$ which is algebraic over $\F_q$,
$\F_q(\alpha)$ is a finite field with $$[\F_q(\alpha): \F_q] =
[K(\alpha):K] \mid n,$$ 
as $K(\alpha)$ is a subfield of $D$.

Let $s$ be a positive divisor of $n$. Put $L_s:=K\F_{q^s}$, the
constant field extension of $K$ of degree $s$. For any place $v$ of
$K$, let $\ell_{s,v}$ be the number of places of $L_s$ over $v$. One
has $\ell_{s,v}=\gcd(s, \deg v)$. The following lemma gives the
criterion for the existence of an embedding of $L_s$ into $D$. 

\begin{lem}\label{22}
There exists an embedding $\iota: L_s \hookrightarrow D$ if and only
if $\ell_{s,v}$ divides $m_v$ for all places $v$ of $K$.
\end{lem}

\begin{proof}
Suppose $D \otimes_K L_s \cong \text{Mat}_{c_s}(\Delta_s)$, where
$\Delta_s$ is a central division algebra over $L_s$.
By \cite[Lemma 2.3]{SYY}, there exists an embedding of $L_s$
into $D$ if and only if $s | c_s$.
For each place $v$ of $K$, let $w$ be a place of $L_s$ lying above
$v$. We denote by $L_{s,w}$ the completion of $L_s$ at $w$; one has
$[L_{s,w}:K_v]=s/\ell_{s,v}$.  Write
$$D \otimes_K L_{s,w} \cong \text{Mat}_{r_w}(\Delta_w^{\prime}),$$
where $\Delta_w^{\prime}$ is a central division algebra over
$L_{s,w}$. Then by
\cite[(31.9) Theorem]{Rei}
(remembering $D_v\simeq \Mat_{m_v}(\Delta_v)$),
\begin{equation}
  \label{eq:21}
  r_w = m_v \cdot \text{gcd}(d_v , [L_{s,w}:K_v]) =
\text{gcd}(m_v d_v, m_v \cdot \frac{s}{\ell_{s,v}})
=\text{gcd}(n, s \cdot \frac{m_v}{\ell_{s,v}}).
\end{equation}
It is known that (cf.\ \cite[(32.17) Theorem]{Rei})
\begin{equation}
  \label{eq:22}
  c_s =
\text{gcd}(r_w: \text{ place $w$ of $L_s$}).
\end{equation}
It follows from (\ref{eq:21}) that  $s|r_w$ if and only if
$\ell_{s,v}|m_v$, and (\ref{eq:22}) says that
$s|c_s$ if and only if $s|r_w$. This completes the proof of the lemma.
\qed
\end{proof}

\begin{defn}\label{defn2.2}
Let $s_0$ be the divisor of $n$ which is maximal such that
$\ell_{s_0,v}$ divides $m_v$ for all places $v$ of $K$. Then the
finite field $\F_D:= \F_{q^{s_0}}$ is called 
the {\it constant field of $D$}. 
\end{defn}

\begin{Rem}
(1) By Lemma~\ref{22}, any maximal finite subfield of $D$ 
is isomorphic to $\F_D$.\\
(2) For our convenience, we also denote by $\F_K$ 
the constant field of $K$.
\end{Rem}

Suppose $n=p_1^{n_1} \dots p_r^{n_r}$, where $p_i$ are distinct prime
numbers and $n_i$ are positive integers. For any $v \in \Sigma_K$, the
set of all places of $K$, let
\[ s=\prod_{i=1}^r p_i^{n_i(s)}, \quad m_v=\prod_{i=1}^r p_i^{m_i(v)},
\quad\text{and} \quad  \gcd(\deg v, n)=\prod_{i=1}^r p_i^{n_i(v)} \]
be the primary decomposition. 
Then one has $\ell_{s,v}\mid m_v$ if and only if
\begin{equation}
  \label{eq:23}
  \min \{n_i(s), n_i(v)\}\le m_i(v), \quad \forall\, i=1,\dots, r.
\end{equation}
Note that $m_i(v)=n_i$ for almost all places $v\in \Sigma_K$.
For each $i=1,\dots, r$, if $n_i(v)\le m_i(v)$ for all 
$v\in \Sigma_K$, then
$n_i(s)$ can be any integer $0\le n_i(s)\le n_i$. Let $S_i:=\{v\in
S\mid n_i(v)>m_i(v)\, \}$. Then if $S_i$ is non-empty, one has
\[ 0\le n_i(s) \le \min_{v\in S_i} \{m_i(v)\}. \]
Therefore, we obtain the following lemma which computes the degree
$s_0=[\F_D:\F_K]$. 

\begin{lem}\label{23}
  Let the notation be as above and let $s_0=\prod_{i=1}^r
  p_i^{n_i(s_0)}$. Then one has for $i=1,\dots,r$,
  \begin{equation}
    \label{eq:24}
   n_i(s_0)=
\begin{cases}
  \min_{v\in S_i} \{m_i(v)\},  &  \text{if $S_i\neq \emptyset$ },\\
  n_i, & \text{otherwise.}
\end{cases}
  \end{equation}
\end{lem}

Now, let $R$ be an $A$-order in $D$. Then the multiplicative group
$R^{\times}$ must be a finite cyclic group. More precisely: 

\begin{lem}\label{21}
Let the notation and assumption be as above and let $R$ be any
$A$-order in $D$. Then
the multiplicative group $R^{\times}$ is isomorphic to
$\F_{q^s}^\times$ for some positive integer $s \mid [\F_D:\F_K]$.

\end{lem}
\begin{proof}
Since $R $ is discrete in $D_{\infty}$ and $R^{\times}$ is in the
maximal compact subring $O_{D_{\infty}}$ of $D_{\infty}$, we have the
finiteness of $R ^{\times}$. 
We show that the group homomorphism
$$ R^{\times} \longrightarrow
(O_{D_{\infty}}/\mathfrak{P}_{\infty})^{\times}.$$ 
is injective.
Let $a \in R^{\times}$ be an element in the kernel.
Then $a = 1 + \alpha$ where $\alpha \in \mathfrak{P}_{\infty}$.
The finiteness of the order of $a$ implies that $\alpha = 0$.
To see this, 
suppose the order of $a = p^r \cdot m$ where $p$ is the characteristic
of $\F_q$ and $\text{gcd}(p,m) = 1$. Write $a^{p^r} = 1+b\Pi$ 
where $b \in O_{\Delta_{\infty}}$ and $\Pi$ is a generator of
$\mathfrak{P}_{\infty}$. 
Then
$$1 = (a^{p^r})^m = 1 + m \cdot b\Pi + \cdots$$
and hence $b = 0$.
Thus we have $a^{p^r} = 1$, i.e.\ $(\alpha)^{p^r} = 0$.
Since $D$ is a division algebra, $\alpha=0$.

Therefore $R^{\times}$ is a finite cyclic group, and
$\F_q[R^\times]\simeq \F_{q^s}$
for some positive integer $s\mid s_0 = [\F_D:\F_q]$. Since
$\F_q[R^\times]\subset R$, one has
\[ R^\times=(\F_q[R^\times])^\times
\simeq \F_{q^s}^\times. \]
This completes the proof of the lemma. 
\qed

\end{proof}

\section{Strategy of computing the class number $h(R)$}
\label{sec:03}

In this section, we explain how to compute the class number $h(R)$ of
a hereditary $A$-order $R$ in $D$.
Recall that $D$ is a definite central simple algebra over $K$ with
$\text{dim}_K D = n^2$ and $R$ is a hereditary $A$-order in $D$.


\subsection{Mass formulas}\label{sec:31}

A \it locally free $($fractional$)$ right ideal $I$ of $R$ \rm is a
projective $A$-lattice in $D$ such that $I \cdot R = I$
and for each finite place $v$ of $K$, there exists $\alpha_v$ in
$D_v^{\times}$ such that $I_v (:= I \otimes_A O_v) = \alpha_v R_v$.
Two locally free right ideals $I_1$ and $I_2$ are called equivalent if
there exists an element $b$ in $D^{\times}$ such that $I_1 = b \cdot
I_2$. We are interested in the number $h = h(R)$ of locally free right
ideal classes of $R$ in $D$.

Let $\mathbb{A}_K^{\infty}$ be the ring of finite adeles of $K$, and
let $\widehat{A}$ be the pro-finite completion of $A$.
Note that the set of locally free right ideal classes of $R$ can be
identified with the finite double coset space
$D^{\times} \backslash \widehat{D}^{\times} / \widehat{R}^\times$,
where $\widehat{D} := D \otimes_K \mathbb{A}_K^{\infty}$ and
$\widehat{R}:= R \otimes_A \widehat{A}$.
More precisely, let $g_1,...,g_h$ be representatives of the double
cosets.
Then
$\{I_i:=D\cap g_i\widehat{R} \mid 1\leq i \leq h\}$
is a set of representatives of locally free right ideal classes of
$R$.

For $1\leq i \leq h$, let $R_i$ be the left order of $I_i$.
Since $D$ is definite, by Lemma~\ref{21} the cardinality of the
multiplicative group 
$R_i^{\times}$ is finite.
The \it mass sum \rm $\text{\rm Mass}(D/K,R)$ is defined by
\begin{equation}
  \label{eq:31}
  \text{\rm Mass}(D/K,R) := \sum_{1\leq i \leq
  h}\frac{1}{\#(R_i^{\times})}.
\end{equation}
Let $S' = S'_R$ be the set of finite places $v$ of $K$ for
which $R_v$ is not a
maximal $O_v$-order in $D_v$.
Recall that $S=S_D$ is the set of ramified places (including $\infty$)
of $K$ for $D$.
The mass sum $\text{Mass}(D/K,R)$ has an explicit description by the
following theorem:

\begin{thm}\label{31}
\text{\rm (Mass formula)}
Let $D$ be a definite central simple algebra over $K$ with $\text{\rm
  dim}_K D = n^2$.
Let $R$ be a hereditary $A$-order in $D$.
For each place $v$ of $K$. Suppose the local invariant $\text{\rm
  inv}_v(D)$ is $\kappa_v/d_v \text{ } \bmod \Z$ and
the invariant of $R$ at $v$ $($when $v$ is a finite place of $K)$ is
$\vec{f}_v = (f_{v,1},...,f_{v,r_v})$. 
Then we have
\begin{equation}
  \label{eq:32}
  \text{\rm Mass}(D/K,R) =
\frac{\#\text{\rm Pic}(A)}{q-1} \cdot \prod_{i=1}^{n-1}\zeta_K(-i)
\cdot \prod\limits_{v \in S} \mathcal{T}_v
\cdot \prod\limits_{v \in S'} \mathcal{T}'_v,
\end{equation}
where $\Pic(A)$ is the Picard group of $A$, $\zeta_K(s)$ is the
Dedekind zeta function of $K$:
$$\zeta_K(s):= \prod_{v \in \Sigma_K} (1-N(v)^{-s})^{-1},$$
the constants $\mathcal{T}_v$ and $\mathcal{T}'_v$ are given
by
$$\mathcal{T}_v = \prod\limits_{1\leq i \leq n-1, \atop d_v \nmid
  i}\big(N(v)^{i}-1\big), $$
and  
\begin{equation}
  \label{eq:325}
  \mathcal{T}'_v := [\GL_{m_v}(O_{\Delta_v}):R_v^\times]=
\frac{\prod\limits_{i=1}\limits^{m_v} \big(N(v)^{d_v
    i}-1\big)}{\prod\limits_{i=1}\limits^{r_v}
  \left(\prod\limits_{j=1}\limits^{f_{v,i}}\big(N(v)^{d_v
      j}-1\big)\right)}.
\end{equation}
\end{thm}

\begin{proof}
  See \cite[p.\ 382]{D-G} and \cite{Pra}. Also see \cite{W-Y} for
  detailed computations of the proof.\qed
\end{proof}

\subsection{Optimal embeddings}\label{sec:32}

Let $\F_D$ be the constant field of $D$. 
We have shown that $R_i^\times\cong \F_{q^{s_i}}^\times$ for
some $s_i\mid s_0$ where $s_0=[\F_D:\F_K]$. 
Set
\begin{equation}
  \label{eq:33}
  h_{s} = h_s(D/K,R):= \#\{ 1\leq i \leq h \mid \#(R_i^{\times}) =
  q^s-1\}. 
\end{equation}
We call $h_s(D/K,R)$ \it the weight-$s$ class number of $R$. \rm 
We point out that $h_s$ depends only on the invariants of $R$ at
finite places of $K$. 
The mass sum $\text{Mass}(D/K,R)$ then can be written as
\begin{equation}
  \label{eq:34}
  \text{Mass}(D/K,R) = \sum_{s \mid s_0} \frac{h_{s}}{q^{s}-1}.
\end{equation}
Note that
\begin{equation}
  \label{eq:35}
  h=\sum_{s|s_0} h_s.
\end{equation}
For any positive divisor $s$ of $n$, we denote by $L_s$ the constant
field extension $K\mathbb{F}_{q^s}$ of degree $s$.
Let $O_{L_s}:=A\mathbb{F}_{q^s}$, the integral closure of $A$ in $L_s$.
For $1\leq i \leq h$, an \it optimal embedding of $O_{L_s}$ into $R_i$
\rm is an embedding $f: L_s \hookrightarrow D$ such that
$$f(L_s)\cap R_i = f(O_{L_s}).$$
If the set of optimal embeddings of $O_{L_s}$ into $R_i$ is non-empty,
then there are exactly $s$ different embeddings into $R_i$. To see
this, the map $f$ is determined by its restriction on the constant
subfield $\F_{q^s}$ and the image $f(\F_{q^s})$ is contained in
the finite field $\F_q[R_i^{\times}]$. 
Therefore, there are exactly $s$ maps $f$. 
Note that an optimal embedding of
$L_s$ into $R_i$ exists if and only if $\F_{q^s}$ can be embedded into
$R_i$.

The number of optimal embeddings of $O_{L_s}$ into $R_i$ are related
to the class numbers
$h_{s'}$ for $s|s'$ via the following identity:
\begin{equation}
  \label{eq:36}
  s\cdot\sum_{s': s\mid s'\mid s_0} h_{s'} = \sum_{1\leq i \leq h} \#\{
\text{optimal embeddings of $O_{L_{s}}$ into $R_i$}\}.
\end{equation}
Let
\begin{equation}
\label{eq:38}
\text{\bf E}(D/K,s,R):= \sum_{1\leq i \leq h} \#\{
\text{optimal embeddings of $O_{L_{s}}$ into $R_i$}\}.
\end{equation}
Then $\text{\bf E}(D/K,s,R)$, as the same as $h$ and $h_s$, also
depends only on the invariants of $R$ at finite places of $K$. 
Suppose we can compute $\text{\bf E}(E/K,s,R)$ for each divisor $s>1$
of $s_0$.  
Then together with the mass formula (\ref{eq:32}) and the equation
(\ref{eq:34}), we have enough equations to solve the numbers $h_s$.  
This gives the class number $h$ in question.

\subsection{Adelization}
\label{sec:33}
Now we focus on the computation of $\text{\bf E}(D/K,s,R)$ for $s>1$. 
Fix an inclusion $\iota: L_{s_0}=K\F_{D}
\hookrightarrow D$.
For each positive divisor $s$ of $s_{0}$, the set of
optimal embeddings of $O_{L_s}$ into $R_i$
can be identified with 
\[ C_{\iota}(L_s)^{\times}\backslash \mathcal{E}_{\iota}(s,R_i)
/(R_i^{\times}), \] 
where $C_{\iota}(L_s)$ is the centralizer of $\iota(L_s)$ in $D$ and
$$\mathcal{E}_{\iota}(s,R_i) := \{g \in D^{\times} \mid g^{-1}
\iota(L_s) g \cap R_i 
= g^{-1}\iota(O_{L_s}) g\}.$$
Set
\begin{equation}
  \label{eq:385}
 \widehat{\mathcal{E}}_{\iota}(s,R):=\{g \in \widehat{D}^{\times} \mid 
 \iota(L_s) \cap g\widehat{R}g^{-1} = \iota(O_{L_s}) \}.
\end{equation}
Then
\begin{lem}\label{32}
We have the following bijection:
$$\begin{tabular}{cccc}
$\Phi:$ & $\coprod\limits_{i=1}^h C_{\iota}(L_s)^{\times}\backslash 
\mathcal{E}_{\iota}(s,R_i) / R_i^{\times}$
 & $\cong$ & $C_{\iota}(L_s)^{\times} \backslash
 \widehat{\mathcal{E}}_{\iota}(s,R) / 
 \widehat{R}^{\times}$\\
 & $g \in \mathcal{E}_{\iota}(s,R_i)$ & $\longmapsto$ &
 $C_{\iota}(L_s)^{\times} g g_i 
 \widehat{R}^{\times}$,
\end{tabular}$$
where $g_1,...,g_h$ are the chosen representatives of double cosets in
$D^{\times} \backslash \widehat{D}^{\times}/ \widehat{R}^{\times}$.
\end{lem}

\begin{proof}
It is clear that $\Phi$ is well-defined.
Now, for each $\hat{g} \in \widehat{\mathcal{E}}_{\iota}(s,R)$, there
exist 
an element $b_{\hat{g}} \in D^{\times}$,
an integer $i_{\hat{g}}$ with $1\leq i_{\hat{g}}
\leq h$, and an element $\hat{\gamma}_{\hat{g}}
\in \widehat{R}^{\times}$ such
that
$$\hat{g} = b_{\hat{g}} \cdot g_{i_{\hat{g}}} \cdot
\hat{\gamma}_{\hat{g}}.$$
Then $b_{\hat{g}}$ must be in
$\mathcal{E}_{\iota}(s,R_{i_{\hat{g}}})$, and 
$$C_{\iota}(L_s)^{\times} \hat{g} \widehat{R}^{\times} \longmapsto
C_{\iota}(L_s)^{\times} b_{\hat{g}} R_{i_{\hat{g}}}^{\times} \in
C_{\iota}(L_s)^{\times} \backslash
\mathcal{E}_{\iota}(s,R_{i_{\hat{g}}}) / 
R_{i_{\hat{g}}}^{\times}$$
gives the inverse map of $\Phi$. \qed
\end{proof}

Let $\widehat{L}_s := L_s \otimes_K \mathbb{A}_K^{\infty}$, and let
$C_{\iota}(\widehat{L}_s)$ be the centralizer of
$\iota(\widehat{L}_s)$ in $\widehat{D}$.
We have the following canonical surjective map
\begin{equation}
  \label{eq:39}
  \Psi: C_{\iota}(L_s)^{\times} \backslash
  \widehat{\mathcal{E}}_{\iota}(s,R) / 
\widehat{R}^{\times} \twoheadrightarrow
    C_{\iota}(\widehat{L}_s)^{\times} \backslash
  \widehat{\mathcal{E}}_{\iota}(s,R) / 
    \widehat{R}^{\times}.
\end{equation}
The fiber of a double coset $C_{\iota}(\widehat{L}_s)^{\times}
\hat{g} \widehat{R}^{\times}$ under the map $\Psi$
is equal to the following double coset space
\begin{equation}
  \label{eq:395}
  C_{\iota}(L_s)^{\times} \backslash 
C_{\iota}(\widehat{L}_s)^{\times}
  / 
\big(C_{\iota}(\widehat{L}_s)^{\times} \cap \hat{g}
\widehat{R}^{\times} 
\hat{g}^{-1}\big).
\end{equation}
Note that the base space
$C_{\iota}(\widehat{L}_s)^{\times} \backslash
\widehat{\mathcal{E}}_{\iota}(s,R) / 
\widehat{R}^{\times}$ 
can be decomposed locally:
\begin{equation}
  \label{eq:310}
  C_{\iota}(\widehat{L}_s)^{\times} \backslash
\widehat{\mathcal{E}}_{\iota}(s,R) / 
\widehat{R}^{\times}=
\prod_{v \neq \infty} C_{\iota}(L_{s,v})^{\times} \backslash
  \mathcal{E}_{v,\iota}(s,R_v) 
/R_v^{\times},
\end{equation}
where
\[ L_{s,v}:= L_s \otimes_K K_v,\quad R_v:= R\otimes_A O_v, \quad
O_{L_s,v} := O_{L_s} \otimes_A O_v=\prod_{w|v} O_{L_s,w},\]
and
\begin{equation}
  \label{eq:311}
  \mathcal{E}_{v,\iota}(s,R_v):= \{g_v \in D_v^{\times} \mid
  \iota(L_{s,v}) \cap g_v R_v g_v^{-1} = \iota(O_{L_{s},v}) \}. 
\end{equation}

Therefore, to compute the number of optimal embeddings in question,
we need to
\begin{itemize}
\item give an explicit parametrization of
$C_{\iota}(L_{s,v})^{\times}
\backslash \mathcal{E}_{v,\iota}(s,R_v) /R_v^{\times}$ for each finite
place $v$ of $K$, and 
\item calculate the cardinality of the fiber of each double coset in
  the base space 
$C_{\iota}(\widehat{L}_s)^{\times} \backslash
\widehat{\mathcal{E}}_{\iota}(s,R) / \widehat{R}^{\times}$. 
\end{itemize}

According to the study of local optimal embeddings in
Section~\ref{sec:07}, the double coset space
$C_{\iota}(L_{s,v})^{\times} 
\backslash \mathcal{E}_{v,\iota}(s,R_v) /R_v^{\times}$ can be
understood clearly. Besides, the fiber of a double coset in 
$C_{\iota}(\widehat{L}_s)^{\times} \backslash
\widehat{\mathcal{E}}_{\iota}(s,R) / \widehat{R}^{\times}$ can be
identified with the set of locally free right ideal classes of a
corresponding hereditary $O_{L_s}$-order in the centralizer 
$C_{\iota}(L_s)$. 
When $s>1$ and $s \mid [\mathbb{F}_D:\mathbb{F}_q]$, the algebra 
$C_{\iota}(L_s)$
is again a definite central simple algebra over $L_s$, 
with $[C_{\iota}(L_s): L_s] = (n/s)^2$. 
Repeating this process, we evaluate the class number $h(R)$ eventually.

\section{Main results}\label{sec:04}

\subsection{}
\label{sec:41}

In this section, we analyze the number of global optimal embeddings. 
We establish relations among weight-$s$ class numbers $h_s(R)$ 
and those of smaller central simple subalgebras, for which we call
the ``generalized \lq\lq transfer principle''.  
Then we use these relations to evaluate the class number 
$h(R)$ recursively. 


Let $\vec{\mathbf{f}}:= (\vec{f}_v)_{v \neq \infty}$ where 
for each finite place $v$ of $K$, $\vec{f}_v =
(f_{v,1},...,f_{v,r_v})$ is a vector in $\mathbb{Z}_{> 0}^{r_v}$ with
$\sum_{j=1}^{r_v} f_{v,j} = m_v$, and for almost all $v$ we have
$r_v=1$. 
Let $R= R(D/K,\vec{\mathbf{f}})$ be a hereditary $A$-order in $D$
so that the invariant of $R_v$ is $\vec{f}_v$ at every finite place
$v$.   

Let $s$ be a positive divisor of $s_0 = [\mathbb{F}_D:\mathbb{F}_q]$. 
Recall the equation~(\ref{eq:36}) (also see (\ref{eq:38})) 
$$s \cdot \sum_{s': s \mid s' \mid s_0} 
h_{s'}(D/K,\vec{\mathbf{f}}) =
\text{\bf E}(D/K,s,\vec{\mathbf{f}}),$$ 
where $h_s(D/K,\vec{\mathbf{f}}) = h_s(D/K,R)$ and
$\text{\bf{E}}(D/K,s, \vec{\mathbf{f}}) = \text{\bf{E}}(D/K,s, R)$
are as in Section~\ref{sec:32}.

Before presenting our main result, we recall some notations and
the invariants which we need: 
\begin{itemize}
\item $D$ is a definite central simple algebra over $K$ with $\dim_K D
  = n^2$. 
\item $\kappa_v /d_v \bmod \mathbb{Z}$ is the local invariants of $D$
  at $v$ and 
$m_v = n/d_v$.
\item For each positive divisor $s$ of $n$,
put 
\[ \ell_{s,v} = \text{gcd}(s, \deg v),\quad  \text{and} \quad
t_{s,v} = \text{gcd}(s/\ell_{s,v}, d_v). \] 
\item We denote by $\Sigma_{L_s}^0$ the set of places $w$ of $L_s =
  K\mathbb{F}_{q^s}$ with $w \nmid \infty$. 
\end{itemize}
Note that $\ell_{s,v}$ is the number of places of $L_s$ over $v$ and
$t_{s,v}$ is the capacity of the central simple algebra
$\Delta_v\otimes_{K_v} L_{s,w}$ over $L_{s,w}$, that is,
$\Delta_v\otimes_{K_v} L_{s,w}=\Mat_{t_{s,v}}(\Delta_{w'})$.   

\subsection{The index set $\Omega(D/K,s,\vec{\mathbf{f}})$}
 \label{sec:42}

Let $\vec{\mathbf{f}} = (\vec{f}_v)_{v \neq \infty}$ be as above, and  
let $s$ be a positive divisor of $s_0 = [\mathbb{F}_D:\mathbb{F}_q]$.
We define a set $\Omega (D/K, s,\vec{\mathbf{f}})$ as 
the product over all finite
places of sets $\Omega_v(D/K,s,\vec{f}_v)$:
\begin{equation}
  \label{eq:41}
  \Omega (D/K, s,\vec{\mathbf{f}})=\prod_{v\neq \infty}
  \Omega_v(D/K,s,\vec{f}_v). 
\end{equation}
For each $v\in \Sigma_K^0$, 
let $\Omega_v(D/K,s,\vec{f}_v)$ denote 
the set consisting of all tuples $(\vec{f}_{w,*})_{w|v}$ indexed by
places $w$ of $L_s$ over $v$, 
where each $\vec{f}_{w,*}=(f_{w,(i,j)})$ is an $r_v\times
t_{s,v}$-matrix with  
non-negative integer entries $f_{w,(i,j)}\in \Z_{\ge 0}$ 
(for $1\le i\le r_v$ and $1\le j \le t_{s,v}$), that satisfy
the following conditions:  
If one puts 
\begin{equation}
  \label{eq:42}
f_{w,i}:=\frac{s}{\ell_{s,v}
    t_{s,v}}\cdot \sum\limits_{j=1}\limits^{t_{s,v}} f_{w,(i,j)},  
\end{equation}
then  
\begin{equation}
  \label{eq:43}
  \sum\limits_{i=1}\limits^{r_v} f_{w,i} = \frac{m_v}{\ell_{s,v}}
\quad \forall\, w \mid v, \text{ and } 
\sum\limits_{w\mid v} f_{w,i} = f_{v,i}\quad 
\text{for } 1\leq i \leq r_v.
\end{equation}



Each element of $\Omega (D/K, s,\vec{\mathbf{f}})$ is also of the form 
$\vec{\mathbf{f}}_* = (\vec{f}_{w,*})_{w
  \in \Sigma_{L_s}^0}$, where $\vec{f}_{w,*}=(f_{w,(i,j)})$ with
non-negative integers $f_{w,(i,j)}$ 
that satisfy the above conditions. We put 
an order on the index set 
$\{(i,j) \}_{1\leq i \leq r_{v}, 1\leq j \leq t_{s,v}}$ by 
$$(i,j) < (i',j') \text{ if } 
\begin{cases} 
  j<j', \text{ or }& \text{} \\ 
  {j=j' \text{ and } i < i',} & \text{ } 
\end{cases}$$ 
and write $\vec{f}_{w,*}$ as a long vector in $\Z_{\ge 0}^{r_v
  \cdot t_{s,v}}$, 
that is, 
$$ \vec{f}_{w,*} = (f_{w,(1,1)}, ...,
f_{w,(r_v,1)},f_{w,(1,2)},..., f_{w,(r_v,2)}, ...,
f_{w,(1,t_{s,v})},...,f_{w,(r_v,t_{s,v})}).$$

When $d_v = 1$ and $r_v = 1$, one gets $t_{s,v} = 1$
and $\vec{f}_v = (n)$. In this case, we get $\vec{f}_{w,*} = (n/s)$
for all places $w$ of $L_s$ lying above 
$v$. Therefore, $$\#(\Omega_v(D/K,s,\vec{f}_v)) = 1.$$
This shows that $\Omega(D/K,s,\vec{\mathbf{f}})$ is a finite
set. 
\begin{lem}\label{41}
The set 
$\Omega_v(D/K,s,\vec{f}_v)$ is non-empty if and only if
\begin{equation}
  \label{eq:44}
  \frac{s}{\ell_{s,v} t_{s,v}}\,  {\Big |}\, f_{v,i}
\end{equation}
for $1\leq i \leq r_v$.
\end{lem}
\begin{proof}
It is clear that the non-emptiness of 
$\Omega_v(D/K,s,\vec{f}_v)$ implies
the condition (\ref{eq:44}). 

Note that
$s/\ell_{s,v}t_{s,v}=[\F_{\Delta_w'}:\F_{\Delta_v}]$. 
Since there is an embedding $L_{s,w}\hookrightarrow 
\Mat_{m_v/\ell_{s,v}}(\Delta_v)$, 
the divisibility $(s/\ell_{s,v}t_{s,v})\mid (m_v/\ell_{s,v})$ is
automatically satisfied. Conversely, suppose the condition (\ref{eq:44})
is satisfied. Then the set $\Omega_v(D/K,s,\vec{f}_v)$ 
is non-empty by the next lemma. This proves the lemma. \qed
\end{proof}

\begin{lem}
  Let $(m_1,\dots,m_\ell)$ and $(f_1,\dots, f_r)$ be two sequences of
  non-negative integers satisfying $\sum_{w=1}^\ell m_w=\sum_{i=1}^r
  f_i$. Then the set 
\[ \Omega:=\left \{ (f_{w,i})\in \Z^{\ell}_{\ge 0}\times \Z^{r}_{\ge
  0}\, {\Big |}\   \sum_{w} f_{w,i}=f_i,\ \forall\, i, \ \text{and} \   
\sum_{i} f_{w.i}=m_w \ \forall\, w \right \} \]
is non-empty. 
\end{lem}
\begin{proof}
  We prove this by induction on $\ell$. The statement holds clearly 
  when $\ell=1$. For any $\ell$, choose non-negative 
  integers  $f_{\ell,1}\dots,f_{\ell,r}$ with $f_{\ell, i}\le f_i$ for
  all $i=1,\dots, r$ and $\sum_{i} f_{\ell,i}=m_\ell$. By induction,
  the set for $(m_1,\dots, m_{\ell-1})$ and $(f_1-f_{\ell,1},
  \dots f_r-f_{\ell,r})$ is non-empty and hence the set $\Omega$ is
  non-empty. This proves the lemma. \qed     
\end{proof}


\subsection{Main results}
\label{sec:43}

For each $\vec{\mathbf{f}}_* \in \Omega(D/K, s, \vec{\mathbf{f}})$,
set 
$\vec{\mathbf{f}}_*^o:= (\vec{f}_{w,*}^o)_{w \in \Sigma_{L_s}^0}$ 
where $\vec{f}_{w,*}^o$ is the vector obtained by removing the zero
entries of the vector $\vec{f}_{w,*}$. For example, if
$\vec{f}_{w,*}=(5,0,4,1,0,1)$, then $\vec{f}_{w,*}^o=(5,4,1,1)$.
See (\ref{eq:38}) for the definition of the 
term $\text{\bf{E}}(D/K,s, \vec{\mathbf{f}})$. 


\begin{thm}\label{43}
Let $D$ be a definite central simple algebra of degree $n^2$ over
$K$. 
For each place $v$ of $K$, the local invariant $\text{\rm inv}_v (D)$
is denoted by $\kappa_v/d_v \bmod \mathbb{Z}$ and $n= m_v d_v$.  
Let $s$ be a positive divisor of $s_0=[\mathbb{F}_D:\mathbb{F}_K]$ and
$\vec{\mathbf{f}}= (\vec{f}_v)_{v \neq \infty}$, where 
for each finite place $v$ of $K$, $\vec{f}_v =
(f_{v,1},...,f_{v,r_v})$ is a vector in $\mathbb{Z}_{> 0}^{r_v}$ with
$\sum_{i=1}^{r_v} f_{v,i} = m_v$, and for almost all $v$ we have
$r_v=1$. 
Then 
\begin{equation}
  \label{eq:45}
  \text{\bf{E}}(D/K,s, \vec{\mathbf{f}}) =  
\sum_{\vec{\mathbf{f}}_* \in \Omega(D/K, s, \vec{\mathbf{f}})}
h(D'_s/L_s, \vec{\mathbf{f}}_*^o).
\end{equation}
Here
$D_s'$ is the centralizer $C_{\iota}(L_s)$ in $D$ $($for an arbitrary
fixed embedding $\iota$ of $L_{s_0}$ to $D)$, 
which is a definite central simple algebra over $L_s$,
and $h(D'_s/L_s, \vec{\mathbf{f}}_*^o)$ is the class number of
hereditary $O_{L_s}$-order $R(D'_s/L_s,\vec{\mathbf{f}}_*^o)$ in
$D'_s$. 
\end{thm}

\begin{proof}
Fix an embedding $\iota: L_{s_0}\hookrightarrow D$.
For any positive divisor $s$ of $s_0$, we must have 
$\ell_{s,\infty} = 1$. 
Therefore $D_s'$ is definite (with respect to the unique place of $L_s$
lying above $\infty$).
 
Let $R = R(D/K,\vec{\mathbf{f}})$ be a hereditary $A$-order in $D$
such that 
the invariant of $R$ is $\vec{\mathbf{f}}$. 
By Lemma~\ref{32}, we have
$$\text{\bf{E}}(D/K,s, \vec{\mathbf{f}}) = \# \Big(
C_{\iota}(L_s)^{\times} \backslash
\widehat{\mathcal{E}}_{\iota}(s,R)/\widehat{R}^{\times}\Big).$$ 
Consider the canonical surjective map
$$ \Psi: C_{\iota}(L_s)^{\times} \backslash
\widehat{\mathcal{E}}_{\iota}(s,R)/\widehat{R}^{\times}  
\twoheadrightarrow C_{\iota}(\widehat{L}_s)^{\times} \backslash
\widehat{\mathcal{E}}_{\iota}(s,R) / \widehat{R}^{\times}.$$ 
Note that
\begin{equation}
  C_{\iota}(\widehat{L}_s)^{\times} \backslash
\widehat{\mathcal{E}}_{\iota}(s,R) / 
\widehat{R}^{\times}=
\prod_{v \neq \infty} C_{\iota}(L_{s,v})^{\times} \backslash
  \mathcal{E}_{v,\iota}(s,R_v) 
/R_v^{\times},
\end{equation}
By Lemma~\ref{72}, Propositions~\ref{75} and \ref{76} (1), and 
Theorem \ref{710} (1), we have a natural bijection
\[ C_{\iota}(L_{s,v})^{\times} \backslash
  \mathcal{E}_{v,\iota}(s,R_v) 
/R_v^{\times} \simeq \Omega_v(D/K,s,\vec{f}_v). \]
From (\ref{eq:41}) we have a natural bijection
\begin{equation}
  \label{eq:46}
  C_{\iota}(\widehat{L}_s)^{\times} \backslash
\widehat{\mathcal{E}}_{\iota}(s,R) / \widehat{R}^{\times}
\simeq \Omega(D/K, s,\vec{\mathbf{f}}).
\end{equation}

Let $[\hat{g}]$ be the double coset corresponding to a given
$\vec{\mathbf{f}}_* \in \Omega(D/K,s,\vec{\mathbf{f}})$. By Propositions
\ref{75}, \ref{76} (2), and Theorem \ref{710} (2), we have 
$$C_{\iota}(L_s) \cap \hat{g}\widehat{R} \hat{g}^{-1}= R(D'_s/L_s,
\vec{\mathbf{f}}_*^o),$$ 
which is a hereditary $O_{L_s}$-order in $D'_s$.
Therefore 
\begin{equation}
  \label{eq:465}
  \begin{split}
   \#\Psi^{-1}([\hat g]) & =\#\bigg(C_{\iota}(L_s)^{\times}\backslash
C_{\iota}(\widehat{L}_s)^{\times}/
\big(C_{\iota}(\widehat{L}_s)^{\times} 
\cap \hat{g} \widehat{R}^{\times} \hat{g}^{-1}\big)\bigg) \\ 
 & = h(D'_s/L_s, \vec{\mathbf{f}}_*^o). 
  \end{split}
\end{equation}
$$ $$ 
This completes the proof of the theorem. \qed
\end{proof}

\begin{thm}\label{435}
  Notations being as above. There is an optimal embedding of
  $O_{L_s}$ into a hereditary order $R'$ in $D$ of invariant
  $\vec{\bff}$ if and only if for all $v\in \Sigma_K^0$, one has
\begin{equation}
  \label{eq:47}
  \frac{s}{\ell_{s,v} t_{s,v}}\, {\Big |}\, f_{v,i}
\end{equation}
for $1\leq i \leq r_v$.
\end{thm}
\begin{proof}
  This follows from (\ref{eq:46}) and Lemma~\ref{41}. \qed
\end{proof}

It is clear that for each positive divisor $s$ of
$s_0=[\mathbb{F}_D:\mathbb{F}_q]$, one has $$[\mathbb{F}_{D_s'} :
\mathbb{F}_{q^s}] = s_0/s.$$ 
Theorem~\ref{43} (and (\ref{eq:36})) tells us that
\begin{equation}
    \label{eq:476}
    \begin{split}
s \cdot \sum_{s': s|s'| s_0} h_{s'}(D/K,\vec{\mathbf{f}})
&= \sum_{\vec{\mathbf{f}}_* \in \Omega(D/K,s,\vec{\mathbf{f}})}
h(D_s'/L_s,\vec{\mathbf{f}}_*^o)  \\
&=\sum_{s': s|s'|s_0}\left(\sum_{\vec{\mathbf{f}}_* \in
    \Omega(D/K,s,\vec{\mathbf{f}})} h_{s'/s}
  (D_s'/L_s,\vec{\mathbf{f}}_*^o) \right).         
    \end{split}
\end{equation}


The following is one of main theorems of this paper, which is a
refinement of the relation (\ref{eq:476}).

\begin{thm}\label{44}
\text{\rm (The generalized transfer principle)}
For any two positive divisors $s$ and $s'$ of $s_0$ 
with $s \mid s'$, we
have  
\begin{equation}
  \label{eq:48}
  s \cdot h_{s'}(D/K,\vec{\mathbf{f}}) = \sum_{\vec{\mathbf{f}}_* \in
  \Omega(D/K,s,\vec{\mathbf{f}})}
h_{s'/s}(D_s'/L_s,\vec{\mathbf{f}}_*^o).
\end{equation}
\end{thm}

\begin{proof}
For positive divisors $s$ and $s'$ of $s_0$ with $s \mid s'$, one
observes that 
$$\Omega(D/K,s', \vec{\mathbf{f}})
= \coprod_{\vec{\mathbf{g}}_* \in \Omega(D/K,s,\vec{\mathbf{f}})}
\Omega(D_s'/L_s,\frac{s'}{s}, \vec{\mathbf{g}}_*^o).$$ 
Suppose the statement holds in the case  $s = s'$, that is, the relation
\begin{equation}
\label{eq4.1}
s \cdot h_{s}(D/K,\vec{\mathbf{f}}) = \sum_{\vec{\mathbf{f}}_* \in
  \Omega(D/K,s,\vec{\mathbf{f}})} h_{1}(D_s'/L_s,\vec{\mathbf{f}}_*^o) 
\end{equation}
holds for any positive divisor $s$ of $s_0$.
Then
$$ s \cdot h_{s'}(D/K,\vec{\mathbf{f}}) 
= \frac{s}{s'} \cdot \left(s'\cdot
  h_{s'}(D/K,\vec{\mathbf{f}})\right) 
= \frac{s}{s'} \cdot \sum_{\vec{\mathbf{f}}_* \in
  \Omega(D/K,s',\vec{\mathbf{f}})}
h_{1}(D_{s'}'/L_{s'},\vec{\mathbf{f}}_*^o).$$ 
On the other hand, one has
\begin{equation}\nonumber
\begin{split}
  \sum_{\vec{\mathbf{g}}_* \in \Omega(D/K,s,\vec{\mathbf{f}})}
& h_{s'/s}(D_s'/L_s,\vec{\mathbf{g}}_*^o) \\
& = \frac{s}{s'} \sum_{\vec{\mathbf{g}}_* \in
  \Omega(D/K,s,\vec{\mathbf{f}})} 
\left(\sum_{\vec{\mathbf{f}}_* \in \Omega(D_s'/L_s, s'/s,
    \vec{\mathbf{g}}_*^o)} h_{1}
  (D_{s'}'/L_{s'},\vec{\mathbf{f}}_*^o)\right).    
\end{split}
\end{equation}


Therefore to complete the proof, it suffices to prove the equality
(\ref{eq4.1}). 
We prove this by induction on the number $\mu(D/K,s):=\sum_{i} n_i$,
where $[\mathbb{F}_D:\mathbb{F}_K]/s=\prod_{i} p_i^{n_i}$ is the prime
decomposition. 

Note that $[\mathbb{F}_{D_s'}:\mathbb{F}_{L_s}] = s_0/s$ for any
positive divisor $s$ of $s_0 = [\mathbb{F}_D:\mathbb{F}_K]$.  
Therefore (\ref{eq4.1}) holds for $\mu(D/K,s) = 0$, i.e.\ $s = s_0 =
[\mathbb{F}_D:\mathbb{F}_K]$. Indeed, by Theorem \ref{43} we have 
\begin{equation}\nonumber
  \begin{split}
  s_0 \cdot h_{s_0}(D/K,\vec{\mathbf{f}}) & = \text{\bf
  E}(D/K,s_0,\vec{\mathbf{f}}) \\
& = \sum_{\vec{\mathbf{f}}_* \in
  \Omega(D/K,s_0,\vec{\mathbf{f}})}
  h(D_{s_0}'/L_{s_0},\vec{\mathbf{f}}_*^o) \\
& = \sum_{\vec{\mathbf{f}}_* \in \Omega(D/K,s_0,\vec{\mathbf{f}})}
h_1(D_{s_0}'/L_{s_0},\vec{\mathbf{f}}_*^o).  
  \end{split}
\end{equation}

Suppose the equality (\ref{eq4.1}) holds 
for any pair $(D''/K'',s'')$ with
$$\mu(D''/K'',s'')<\mu(D/K,s).$$ 
Then we get
\begin{eqnarray}
\label{eq4.2} \text{\bf E}(D/K,s,\vec{\mathbf{f}})
&=& s \cdot \sum_{s': s \mid s' \mid s_0} h_{s'}(D/K,\vec{\mathbf{f}})
 \\ 
 &=& s \cdot h_s(D/K,\vec{\mathbf{f}}) + 
 \sum_{s': s \mid s' \mid s_0 \atop s< s'} \frac{s}{s'} \left(s' \cdot
 h_{s'}(D/K,\vec{\mathbf{f}})\right). \nonumber 
\end{eqnarray}
Since $\mu(D/K,s')<\mu(D/K,s)$ for any positive divisor $s'$ of $s_0$
with $s|s'$ and $s< s'$, by induction hypothesis, one has 
\begin{equation}
\label{eq4.3}
\sum_{s': s \mid s' \mid s_0 \atop s< s'} \frac{s}{s'} \left(s' \cdot
  h_{s'}(D/K,\vec{\mathbf{f}})\right) =  
\sum_{s': s \mid s' \mid s_0 \atop s< s'} \frac{s}{s'}
\left(
\sum_{\vec{\mathbf{f}}_* \in \Omega(D/K,s',\vec{\mathbf{f}})} 
h_1(D_{s'}'/L_{s'},\vec{\mathbf{f}}_*^o)\right).
\end{equation}
On the other hand, by Theorem~\ref{43} we get
\begin{eqnarray}
\label{eq4.4} & &\text{\bf E}(D/K,s,\vec{\mathbf{f}}) \\
&=& \sum_{\vec{\mathbf{f}}_* \in \Omega(D/K, s, \vec{\mathbf{f}})}
\left(\sum_{s': s \mid s' \mid s_0} h_{s'/s}(D'_s/L_s,
  \vec{\mathbf{f}}_*^o)\right) \nonumber \\ 
&=& \sum_{\vec{\mathbf{f}}_* \in \Omega(D/K, s, \vec{\mathbf{f}})}
h_{1}(D'_s/L_s, \vec{\mathbf{f}}_*^o) + 
\sum_{s': s \mid s' \mid s_0, \atop s<s'}
\left(\sum_{\vec{\mathbf{g}}_* \in \Omega(D/K, s, \vec{\mathbf{f}})}
  h_{s'/s}(D'_s/L_s, \vec{\mathbf{g}}_*^o)\right). \nonumber  
\end{eqnarray}
Since $\mu(D'_s/L_s, s'/s) < \mu(D/K,s)$ for any positive divisor $s'$
of $s_0$ with $s|s'$ and $s< s'$, by induction hypothesis again we
obtain that 
\begin{eqnarray}
\label{eq4.5}& & \sum_{s': s \mid s' \mid s_0, \atop s<s'}
\left(\sum_{\vec{\mathbf{g}}_* \in \Omega(D/K, s, \vec{\mathbf{f}})}
  h_{s'/s}(D'_s/L_s, \vec{\mathbf{g}}_*^o)\right) \\ 
&=&
\sum_{s': s \mid s' \mid s_0, \atop s<s'}\frac{s}{s'}\cdot
  \left(\sum_{\vec{\mathbf{g}}_* \in \Omega(D/K, s,
  \vec{\mathbf{f}})}\left[ 
\sum_{\vec{\mathbf{f}}_* \in
  \Omega(D_s'/L_s,s'/s,\vec{\mathbf{g}}_*^o)} 
h_1(D_{s'}'/L_{s'},\vec{\mathbf{f}}_*^o)\right]\right) \nonumber \\ 
&=& \sum_{s': s \mid s' \mid s_0 \atop s< s'} \frac{s}{s'}
\left(
\sum_{\vec{\mathbf{f}}_* \in \Omega(D/K,s',\vec{\mathbf{f}})}
h_1(D_{s'}'/L_{s'},\vec{\mathbf{f}}_*^o)\right). \nonumber 
\end{eqnarray}
Then the equation (\ref{eq4.1}) follows from (\ref{eq4.2}) --
(\ref{eq4.5}).
This completes the proof of the theorem. \qed
\end{proof}

Recall that
$$\text{Mass}(D/K,\vec{\mathbf{f}}) = \sum_{s \mid
  [\mathbb{F}_D:\F_K]} \frac{h_s(D/K,\vec{\mathbf{f}})}{q^s-1}.$$ 
Theorem~\ref{44} states that one can compute the weight-$s'$ class
number $h_{s'}(D/K,\vec{\mathbf{f}})$ in terms of the weight-$s'/s$
class numbers of a central simple algebra $D'_s/L_s$ of smaller
degree. Then by induction one can compute the weight-$s'$ class
number $h_{s'}(D/K,\vec{\mathbf{f}})$. However, since the index set
$\Omega(D/K,s,\vec{\mathbf{f}})$ could be quite large, the computation
by induction would be very complicated. The following theorem states
that we can compute directly a modified version of the right hand side
of the equation (\ref{eq:48}) in Theorem~\ref{44} 
without going through the induction step,
but keep the same information so that we can compute the weight-$s'$
class number $h_{s'}(D/K,\vec{\mathbf{f}})$. This simplifies the
computation significantly.    

\begin{thm}\label{45}
Notations being as above, one has
\begin{equation}
  \label{eq:414}
  s \cdot \sum_{s': s \mid s' \mid [\F_D:\F_K]}
\frac{h_{s'}(D/K,\vec{\mathbf{f}})}{q^{s'}-1} 
 = \sum_{\vec{\mathbf{f}}_* \in \Omega(D/K,s,\vec{\mathbf{f}})}
\text{\rm Mass}(D_s'/L_s,\vec{\mathbf{f}}_*^o).
\end{equation}
\end{thm}
\begin{proof}
This is obtained by multiplying the factor 
$(q^{s'}-1)^{-1}$ on (\ref{eq:48})
and summing over all positive integers $s'$ with $s|s'|s_0$. \qed  
\end{proof}

Set $\text{Mass}(D/K):=
\text{Mass}(D/K,\vec{\mathbf{f}}_{\text{max}})$ where
$\vec{\mathbf{f}}_{\text{max}}$ is the invariant of a maximal
$A$-order in $D$. 
Then the mass formula in Theorem~\ref{31} says that
$$\text{Mass}(D/K,\vec{\mathbf{f}}) = \text{Mass}(D/K) 
  \cdot \prod_{v  \in \Sigma_K^0} 
\mathcal{T'}_v(D/K,\vec{f}_v),$$ 
where for a vector $\vec{h}_v = (h_{v,1},...,h_{v,r_v}) \in
\mathbb{Z}_{\geq 0}^{r_v}$, we set
\begin{equation}
  \label{eq:418}
  \mathcal{T'}_v(D/K,\vec{h}_v):=
\frac{\prod\limits_{i=1}\limits^{m_v} \big(N(v)^{d_v
    i}-1\big)}{\prod\limits_{i=1}\limits^{r_v}
  \left(\prod\limits_{j=1}\limits^{h_{v,i}}\big(N(v)^{d_v
      j}-1\big)\right)}.
\end{equation}
For each finite place $v$ of $K$, define
\begin{equation}
  \label{eq:419}
  \Theta_v (D/K,s,\vec{f}_v):= \sum_{(\vec{f}_{w,*})_{w \mid v} \in
  \Omega_v(D/K,s,\vec{f}_v)} \left(\prod_{w \mid v}
  \mathcal{T'}_w(D_s'/L_s,\vec{f}_{w,*})\right).
\end{equation}
It is clear that $\Theta_v(D/K,s,\vec{f}_v) = 1$ if $d_v = 1$ and
$\vec{f}_v = (n)$. 
Therefore we can rewrite Theorem~\ref{45} as the following theorem, 
which reduces the computation in purely local terms.

\begin{thm}\label{46}
Notations being as above, one has
\begin{equation}
  \label{eq:420}
  s \cdot \sum_{s': s \mid s' \mid [\F_D:\F_K]}
\frac{h_{s'}(D/K,\vec{\mathbf{f}})}{q^{s'}-1} 
= \text{\rm Mass}(D_s'/L_s) \cdot \prod_{v \in \Sigma_K^0}
\Theta_v(D/K,s,\vec{f}_v).
\end{equation}
\end{thm}




\subsection{Explicit computation of $\Theta_v(D/K,s,\vec{f}_v)$}
In this subsection we give a simple method to compute the term  
$\Theta_v(D/K,s,\vec{f}_v)$ effectively.

Fix a finite place $v$ of $K$.
Recall that $D_v = D \otimes_K K_v \cong \text{Mat}_{m_v}(\Delta_v)$
and $\Delta_v$ is a central division algebra over $K_v$ with
$[\Delta_v:K_v] = d_v^2$. 
For any positive divisor $s$ of $[\F_D:\F_K]$,
recall $\ell_{s,v} := \text{gcd}(s,\deg v)$ and $t_{s,v} :=
\text{gcd}(s/\ell_{s,v}, d_v)$. We have $L_{s,v}=\prod_{w=1}^{\ell_{s,v}}
  L_w$, where each $L_w$ is a unramified extension over $K_v$ of
  degree $s/\ell_{s,v}$, and $\Delta_v\otimes_{K_v} L_w\simeq
  \Mat_{t_{s,v}}(\Delta_w')$. The division algebra $\Delta_w'$ has
  degree $d'_v:=d_v/t_{s,v}$ over $L_w$ and its residue field
  $\F_{\Delta_w'}$ has cardinality $N(v)^{d'_v s/\ell_{s,v}}$.

We assume that the set 
$\Omega_v(D/K,s,\vec{f}_v)$ is non-empty. By Lemma~\ref{41}, this is
equivalent to the condition  
$$ \frac{s}{\ell_{s,v} t_{s,v}}\, {\Big |}\, f_{v,i}, \quad \forall\,1\leq
i \leq r_v.$$ 

Since $s$ divides $[\F_D:\F_K]$, one gets the divisibility
$$\frac{s}{\ell_{s,v} t_{s,v}} \,{\Big |} 
\left( \frac{m_v}{\ell_{s,v}}\right ).$$ 


The completion $D'_s\otimes_{L_s}  L_{s,v}$ of the division algebra
$D_s'$ at $v$ is the centralizer of $L_{s,v}$
in $D_v$, which is isomorphic to 
\begin{equation}
  \label{eq:421}
  \prod_{w=1}^{\ell_{s,v}} \Mat_{m_{v}^{(s)}}(\Delta'_w),
\end{equation}
where
$$m_{v}^{(s)}= \left (\frac{s}{\ell_{s,v} t_{s,v}} \right )^{-1} 
\cdot \frac{m_v}{\ell_{s,v}}= \frac{m_v t_{s,v}}{s}. $$
Set 
$$f_{v,i}^{(s)}:=\left (\frac{s}{\ell_{s,v} t_{s,v}} \right )^{-1}  
 \cdot f_{v,i}\in \bbN, \quad \text{for} \ 
1\leq i \leq r_v.$$ 
Then the set 
$\Omega_v:=\Omega_v(D/K,s,\vec{f}_v)$ consists of all elements 
\[ (f_{w,(i,j)})_{w, i,j}  
\in \Z_{\ge 0}^{\ell_{s,v}}\times 
\Z_{\ge 0}^{r_v}\times \Z_{\ge 0}^{t_{s,v}}
\] 
that satisfy the following conditions
\begin{equation}
  \label{eq:422}
  \sum_{ i,j} f_{w,(i,j)} =
m_v^{(s)}, \text{ }\forall\, 1 \le w \le \ell_{s,v} \quad\text{and} 
\quad \sum_{w, j} f_{w,(i,j)} = f_{v,i}^{(s)}, \text{ }
\forall\,\, 1\leq i \leq r_v.
\end{equation}

Consider the following formal power series
$$F(T) := 1+ a_1 T + a_2 T^2 + \cdots + a_{\nu} T^{\nu} + \cdots \in
\mathbb{Q}[[T]],$$ 
where
\begin{equation}
  \label{eq:423}
a_{\nu}:= 
\prod_{k=1}^{\nu} [N(v)^{(d'_v s/\ell_{v,s})\cdot k} -1]^{-1}.  
\end{equation}

Put 
\[ 
G(\ul X, \ul Y, \ul Z):=\prod_{w=1}^ {\ell_{s,v}}\prod_{i=1}^{r_v} 
\prod_{j=1}^{t_{s,v}} 
F(X_w \cdot Y_i \cdot Z_j) \in \mathbb{Q}[[\ul X,
  \ul Y, \ul Z ]], \]
where $\ul X=(X_1,\dots,X_{\ell_{s,v}})$, $\ul Y=( Y_1,\dots,Y_{r_v})$,
  and $\ul Z=(Z_1,\dots,Z_{t_{s,v}})$. 
Then we can use the generating function $G$ to compute the term 
$\Theta_v(D/K,s,\vec{f}_v)$.

\begin{prop}\label{48}
The coefficient of the monomial
$$X_1^{m_v^{(s)}}\cdots X_{\ell_{s,v}}^{m_v^{(s)}}
Y_1^{f_{v,1}^{(s)}}\cdots Y_{r_v}^{f_{v,r_v}^{(s)}}$$ 
of the formal power series $G(\ul X, \ul Y, 1,...,1) \in
\mathbb{Q}[[\ul X, \ul Y]]$ is equal to 
$$\left[\prod_{k=1}^{m_v^{(s)}}[N(v)^{(d'_v s/\ell_{v,s})\cdot k} -1]
\right]^{-\ell_{s,v}} \cdot \Theta_v(D/K,s,\vec{f}_v).$$
\end{prop}

\begin{proof}
First for each $w=1,\dots, \ell_{s,v}$, 
$$D_s'\otimes_{L_s} L_{w} \cong
\text{Mat}_{m_v^{(s)}}(\Delta_{w}')$$ 
where $\Delta_{w}'$ is a central division algebra over $L_{w}$ of
degree $d_v'$. 
Using the definition of $\Theta_v$ 
we can express the term $\Theta_v(D/K,s,\vec{f}_v)$ 
as follows
$$\sum_{(f_{w,(i,j)}) \in \Omega_v} 
\left[\prod_{w=1}^{\ell_{s,v}} \left( 
\frac{\displaystyle \prod_{k=1}^{
        m_v^{(s)}}[N(v)^{(d'_v s/\ell_{v,s})\cdot k} -1]}
{\displaystyle
\prod_{i=1}^{r_v} \prod_{j=1}^{t_{s,v}} \prod_{k=1}^{f_{w,(i,j)}}
      [N(v)^{(d'_v s/\ell_{v,s})\cdot k} -1]}\right)\right].$$ 
Therefore the term
$$\left[\prod_{k=1}^{m_v^{(s)}}[N(v)^{(d'_v s/\ell_{v,s})\cdot k} -1]
\right]^{-\ell_{s,v}} 
\cdot \Theta_v(D/K,s,\vec{f}_v)$$
is equal to
\begin{equation}
\label{eqn4.5.1}
\sum_{(f_{w,(i,j)})_{w,i,j}, \text{ with (\ref{eq:422})}} 
\left[\prod_{w=1}^{\ell_{s,v}} \prod_{i=1}^{r_v} \prod_{j=1}^{t_{s,v}}
  a_{f_{w,(i,j)}} \right]. 
\end{equation}

On the other hand, one sees that
the coefficient of the monomial 
 $$X_1^{m_v^{(s)}}\cdots X_{\ell_{s,v}}^{m_v^{(s)}}
 Y_1^{f_{v,1}^{(s)}}\cdots Y_{r_v}^{f_{v,r_v}^{(s)}}$$ 
of the formal power series
$G(\ul X, \ul Z, 1,\dots ,1)$
is equal to
\begin{equation}
\label{eqn4.5.2}
 \sum_{(f_{w,i,j})_{w,i,j} \text{ with (\ref{eq:422})}} 
 \left[\prod_{w=1}^{\ell_{s,v}} \prod_{i=1}^{r_v} \prod_{j=1}^{t_{s,v}}
a_{f_{w,i,j}} \right]. 
\end{equation}
The result follows from the equations (\ref{eqn4.5.1}) and
(\ref{eqn4.5.2}). This completes the proof of the proposition. \qed
\end{proof}

\subsection{A recursive formula for computing class
  numbers}\label{sec4.1} 
In this subsection, we present an
explicit recursive formula to compute the class numbers
$h_s(D/K,\vec{\mathbf{f}})$ in terms of mass sums.
Recall that $\text{Mass}(D/K,\vec{\mathbf{f}})$ can be expressed in
terms of special zeta values using the mass formula
(Theorem~\ref{31}). Then the target class 
number $h(D/K,\vec{\mathbf{f}})$ is simply the following sum 
$$h(D/K,\vec{\mathbf{f}}) = \sum_{s\in \mathbb{N}, \text{ } s \mid
  [\mathbb{F}_D:\mathbb{F}_K]} h_s(D/K, \vec{\mathbf{f}}).$$ 
We recall that (\ref{defn2.2}) $\mathbb{F}_D$ (resp.\
$\mathbb{F}_K$) is the constant field of $D$ (resp.\ $K$). One
computes the degree $s_0=[\mathbb{F}_D:\mathbb{F}_K]$ using
Lemma~\ref{23}. 

For each positive divisor $s$ of $[\mathbb{F}_D: \mathbb{F}_K]$,
we call the number of prime factors of the integer 
$[\mathbb{F}_D: \mathbb{F}_K]/s$ with multiplicity 
the {\it depth of $(D/K,s)$}, and denote it by $\mu(D/K,s)$. That is, if
$[\mathbb{F}_D: \mathbb{F}_K]/s=\prod_{i} p_i^{n_i}$ is the prime
decomposition, then $\mu(D/K,s)=\sum_i n_i$. \\ 

(1) The case $\mu(D/K,s)=0$, i.e. $s=s_0$. It is clear that
    $[\mathbb{F}_{D_s'}:\mathbb{F}_{L_s}] = 1$. Theorems~\ref{45} 
and \ref{46}
    states that 
\begin{eqnarray}
h_s(D/K,\vec{\mathbf{f}})
&=& \frac{q^s-1}{s} \cdot \sum_{\vec{\mathbf{f}}_* \in
  \Omega(D/K,s,\vec{\mathbf{f}})} \text{Mass}(D'_{s}/L_{s},
\vec{\mathbf{f}}_*^o) \\ 
&=& \frac{q^s-1}{s} \cdot \text{Mass}(D'_s/L_s) \cdot \prod_{v \in
  \Sigma_K^0} \Theta_v(D/K,s,\vec{f}_v). \nonumber  
\end{eqnarray}
Therefore we have evaluated the class number 
$h_s(D/K,\vec{\mathbf{f}})$ when
$\mu(D/K,s)=1$. \\

(2) Given a positive integer $N$, assume 
that the class number 
$h_{s'}(D/K,\vec{\mathbf{f}})$ has been evaluated for any positive
integer $s'$ with $\mu(D/K,s') \leq N$.  
Let $s$ be a positive divisor of $[\mathbb{F}_D:\mathbb{F}_K]$ with
the depth 
$\mu(D/K,s) = N+1$. Theorem~\ref{46} says that
\begin{equation}\label{eq4.7}
s \cdot \sum_{s': s \mid s' \mid [\F_D:\F_K]}
\frac{h_{s'}(D/K,\vec{\mathbf{f}})}{q^{s'}-1} 
= \text{\rm Mass}(D_s'/L_s) \cdot \prod_{v \in \Sigma_K^0}
\Theta_v(D/K,s,\vec{f}_v). 
\end{equation}
Since the right hand side of the equation~(\ref{eq4.7}) can be
computed explicitly, 
one evaluates the term $h_s(D/K,\vec{\mathbf{f}})$ by
\begin{eqnarray}
h_s(D/K,\vec{\mathbf{f}})
&=& \frac{q^s-1}{s} \cdot
\Bigg(\text{\rm Mass}(D_s'/L_s) \cdot \prod_{v \in \Sigma_K^0}
\Theta_v(D/K,s,\vec{f}_v) \nonumber \\ 
& & - s \cdot \sum_{s': s \mid s' \mid [\F_D:\F_K], \atop
  s'>s}\frac{h_{s'}(D/K,\vec{\mathbf{f}})}{q^{s'}-1}\Bigg). \nonumber  
\end{eqnarray}

By the steps (1) and (2), we can compute all class numbers 
$h_s(D/K,\vec{\mathbf{f}})$ and hence compute the desired class number 
$h(D/K,\vec{\mathbf{f}})$. 

\begin{rem}
Since the mass sum can be expressed by integral values of the 
zeta function, the class number 
$h(D/K,\vec{\mathbf{f}})$ can be expressed in terms
of special zeta values eventually. 
\end{rem}

\section{Special cases}\label{sec:05}

In this section, we explain first that Theorem~\ref{44} is
a generalization of Gekeler's transfer principle
\cite{Ge3}. We also deduce an explicit class number formula 
in the case where the 
degree $s_0=[\mathbb{F}_D:\mathbb{F}_K]$ is a prime. Then we give one example
to illustrate how to compute the class number 
using the results in Section~\ref{sec:04}. 

\subsection{Gekeler's transfer principle}\label{sec:51}
 
Let $K,\infty,A$ be as before and $v_0$ a fixed finite place of 
$K$. For any positive integer $n$, let $\Lambda(K,\infty,v_0;n)$
denote 
the set of isomorphism classes of supersingular Drinfeld $A$-modules 
of rank $n$ over $\ol \F_{v_0}$. For any object $\phi$ in
$\Lambda(K,\infty,v_0;n)$, the automorphism group $\Aut(\phi)$ of
$\phi$ 
is isomorphic to $\F_{q^s}^\times$ for some positive integer $s$ 
with $s|n$. 
For any positive divisor $s$ of $n$, let
\[ \Lambda(K,\infty,v_0;n, s):=\{\phi\in \Lambda(K,\infty,v_0;n)\mid
\Aut(\phi)\simeq \F_{q^s}^\times\}. \]

\begin{thm}[The transfer principle \cite{Ge3}]\label{51}
 For any two positive divisors $s$ and $s'$ of $n$ with $s|s'$, there is a
 natural bijection 
  \begin{equation}
    \label{eq:51}
    \Phi:\Lambda(L_s,\infty_s,v_{0,s}; n/s', s'/s)\isoto 
\Lambda(K,\infty, v_0 ; n, s'),
  \end{equation}
where $L_s$ is the constant field extension of degree $s$, $v_{0,s}$
and 
$\infty_s$ are the unique places of $L_s$ over $v_0$ and $\infty$,
respectively. Recall that $O_{L_s}$ is the integral closure of $A$ in
$L_s$.  
The map $\Phi$ sends a Drinfeld $O_{L_s}$-module $\phi_s:O_{L_s}\to
\ol 
\F_{v_{0,s}}\{\tau\}$ to the Drinfeld $A$-module $\phi:=\phi_s|_A$.   
\end{thm}

Let $D$ be the endomorphism algebra $\End(\phi)\otimes_A K$ of $\phi$,
where $\phi$ is an object in $\Lambda(K,\infty, v_0 ; n)$.  
It is a basic fact that (cf. \cite{Ge2}) $$-\text{inv}_{\infty}(D) \equiv
\text{inv}_{v_0}(D) \equiv 1/n \bmod \mathbb{Z}$$ and 
$$\text{inv}_v(D) \equiv 0 \bmod \mathbb{Z} \text{ for } v \neq v_0,
\infty. $$ 
The subset $\Lambda(K,\infty, v_0 ; n, s)$ is non-empty if and only if 
$s|s_0$, where $s_0:=[\F_D:\F_K]$. Using the notation as before,
$m_\infty=m_{v_0}=1$. By Lemma~\ref{22}, 
the integer $s_0$ is the largest divisor $s$
of $n$ that satisfies $(\deg \infty, s)=(\deg v_0,s)=1$. It follows
that  
\begin{equation}
  \label{eq:52}
   s_0=\prod_{p\nmid \deg v_0\cdot \deg \infty} p^{\ord_{p}(n)}.
\end{equation}


Let $R:=\End(\phi)$ be the endomorphism ring of $\phi$, 
where $\phi\in \Lambda(K,\infty, v_0 ; n)$. 
Then $R$ is a maximal $A$-order in $D$ and
$h(D/K,R)=\# \Lambda(K,\infty, v_0 ; n)$.
Moreover, for any positive divisor $s$ of $s_0$ one gets
$$h_s(D/K,R) = \# \Lambda(K,\infty, v_0 ; n, s).$$
Therefore Gekeler's transfer principle asserts that
for $s\mid s' \mid s_0$, one has
$$h_{s'}(D/K,R) = h_{s'/s}(D_s'/L_s,R_s),$$
where $R_s$ is a maximal $O_{L_s}$-order in $D_s'$. 

Now, since $R$ is a maximal $A$-order of $D$, the invariant
$\vec{\mathbf{f}}_R = (\vec{f}_v)_{v \neq \infty}$ of $R$ is as follows 
$$r_v = 1 \text{ and } \vec{f}_v = (m_v), \quad \forall\, 
v\in \Sigma_K^0.$$  
As $D$ is of Drinfeld type, we have
$$m_v = 
\begin{cases} 1 & \text{ if $v = \infty$ or $v_0$,}\\
              n & \text{ otherwise;} 
\end{cases}
\text{ and } d_v = \begin{cases} n & \text{ if $v = \infty$ or
    $v_0$,}\\ 
    1 & \text{ otherwise.} \end{cases} $$
Consider elements  $\vec{\mathbf{f}}_* = (\vec{f}_{w,*})_
{w \in \Sigma_{L_s}^0}$ 
in $\Omega(D/K,s, \vec{\mathbf{f}}_R)$. For $w \neq v_{0,s}$, 
the only possible choice for
$\vec{f}_{w,*}$ for $w \neq v_{0,s}$ is $(n/s)$. 
For $w = v_{0,s}$, 
the vector $\vec{f}_{w,*} = (f_{w,(1,j)})_{1\leq j \leq s}$ satisfies 
the
    following property:
$$f_{w,(1,j)} = 1 \text{ for some $j$ and }
    f_{w,(1,j')} = 0 \text{ for } j' \neq j.$$ 
Therefore:
\begin{itemize}
\item[(i)] The cardinality of $\Omega(D/K,s, \vec{\mathbf{f}}_R)$ is $s$.
\item[(ii)]  For any $\vec{\mathbf{f}}_* \in \Omega(D/K,s,
  \vec{\mathbf{f}}_R)$, the ring $R(D_{s}^{\prime}/L_{s},
  \vec{\mathbf{f}}_*^o)$ is a maximal $O_{L_{s}}$-order $R_s$ in
  $D_s'$. 
\end{itemize}
By Theorem~\ref{44}, we have
\begin{eqnarray}
s\cdot h_{s'}(D/K,R) &= & s\cdot h_{s'}(D/K,\vec{\mathbf{f}}_R)
\nonumber \\ 
&=& \sum_{\vec{\mathbf{f}}_* \in \Omega(D/K,s,\vec{\mathbf{f}}_R)}
h_{s'/s}(D_s'/L_s,\vec{\mathbf{f}}_*^o) \nonumber \\ 
&=& s \cdot h_{s'/s}(D_s'/L_s, R_s). \nonumber
\end{eqnarray}
This is exactly Gekeler's transfer principle.

\subsection{The case where  
$[\mathbb{F}_D:\mathbb{F}_K]$ is a
  prime}\label{sec5.2} 

Assume that $s_0 = [\mathbb{F}_D:\mathbb{F}_K]$ is a prime number. 
For any hereditary $A$-order $R(D/K,\vec{\mathbf{f}})$, one has
$$h(D/K,\vec{\mathbf{f}}) = h_1(D/K,\vec{\mathbf{f}}) +
h_{s_0}(D/K,\vec{\mathbf{f}}),$$ 
and
$$\text{Mass}(D/K,\vec{\mathbf{f}}) =
\frac{h_1(D/K,\vec{\mathbf{f}})}{q-1} +
\frac{h_{s_0}(D/K,\vec{\mathbf{f}})}{q^{s_0}-1}.$$ 
By Theorem~\ref{46}, we have
\begin{eqnarray}
h_{s_0}(D/K,\vec{\mathbf{f}}) 
&=&\frac{q^{s_0}-1}{s_0} \cdot 
\sum_{\vec{\mathbf{f}}_*\in \Omega(D/K,s_0,\vec{\mathbf{f}})}
\text{Mass}(D_{s_0}'/L_{s_0}, \vec{\mathbf{f}}_*^o). \nonumber
\end{eqnarray}
Therefore we get the following result.

\begin{thm}\label{52}
When $s_0=[\mathbb{F}_D:\mathbb{F}_K]$ is a prime number, we have
\begin{eqnarray}\label{eq:53}
h(D/K,\vec{\mathbf{f}})
&=& (q-1)\cdot \text{\rm Mass}(D/K,\vec{\mathbf{f}}) \nonumber \\
& &+ \frac{q^{s_0}-q}{s_0} \cdot 
\text{\rm Mass}(D_{s_0}'/L_{s_0}) \cdot \prod_{v \in \Sigma_K^0}
\Theta_v(D/K,s_0,\vec{f}_v). 
\end{eqnarray}
\end{thm}

Now, suppose $n$ is a prime number (where $[D:K] = n^2$).
For each place $v$ of $K$, let
$$\epsilon(v) := 
\begin{cases} 1 & \text{ if $n \nmid \deg v$} \\
0 & \text{ otherwise.} 
\end{cases}$$
There exists an embedding of $L_{n}$ into $D$ if and only if 
$\prod_{v \in S_D} \epsilon(v) = 1$. 
Since $D_{n}^{\prime} \cong L_{n}$ and $A\mathbb{F}_{q^{n}} =
O_{L_{n}}$, 
we have
\begin{equation}
  \label{eq:54}
  {\rm Mass}(D_{n}'/L_{n}) 
= \frac{\#(\text{Pic}(O_{L_{n}}))}{(q^{n}-1)}
\end{equation}
and 
\begin{equation}
  \label{eq:55}
  \Theta_v(D/K,n,\vec{f}_v) = \#(\Omega_v(D/K,n,\vec{f}_v)).
\end{equation}
Therefore it suffices to compute 
$\#(\Omega_v(D/K,n, \vec{f}_v))$ for the place $v$ where $D_v$ is a 
division algebra or where the order $R_v$ is not maximal. 

Let $S'$ be the set of finite places of $K$ where $R$ is not maximal. 
Since $n$ is a prime, the intersection of $S_D$ and $S'$ is empty. 
Suppose $v \in S_D-\{\infty\}$.
For $(\vec{f}_{w,*})_{w \mid v} \in
\Omega_v(D/K,n, \vec{f}_v)$, there are precisely
$n$ choices for $\vec{f}_{w,*}$ if $w$ is the only one place lying
above $v$ where $v \in S_D-\{\infty\}$, i.e.\
\begin{equation}
  \label{eq:56}
  \#(\Omega_v(D/K,n,\vec{f}_v)) = n.
\end{equation}
Let $v \in S'$. If $\ell_{n,v} = \text{gcd}(n,\deg v) = 1$, then
$\Omega_v(D/K,n, \vec{f}_v)$ is empty. 
Suppose $\ell_{n,v} = n$ for all $v \in S'$. Then $t_{n,v} = 1$ and
${m_v}/{\ell_{n,v}} = 1$ for $v \in S'$. 
Fix a place $v \in S'$. Then for places
$w_1,...,w_{n}$ lying above $v$, one has
$$\vec{f}_{w_{\nu},*} = (f_{w_{\nu}, (i,1)})_{1\leq i \leq r_v}$$
with $f_{w_{\nu},(i,1)} = 1$ for some $i$ and $f_{w_{\nu},(i',1)} =
0$ for $i' \neq i$, and  
$$\sum_{\nu=1}^{n} f_{w_{\nu},(i,1)} = f_{v,i} \text{ for } 1\leq i
\leq r_v.$$ 
Therefore the number of choices of
$(\vec{f}_{w_1,*},...,\vec{f}_{w_{n},*}) \in \Omega_v(D/K,n,\vec{f}_v)$,
i.e.\ the cardinality of $\Omega_v(D/K,n,\vec{f}_v)$,
is
\begin{equation}
  \label{eq:57}
  \frac{n!}{f_{v,1}!\cdots f_{v,r_v}!}.
\end{equation}

We conclude that $\Omega(D/K,n, \vec{\mathbf{f}})$
is non-empty if and only if
\begin{equation}
  \label{eq:58}
  \prod_{v \in S_D} \epsilon(v) \cdot
\prod_{v \in S'} (1-\epsilon(v)) = 1,
\end{equation}
and in this case, 
one has
\begin{equation}
  \label{eq:59}
  \#(\Omega(D/K,n, \vec{\mathbf{f}})) = \prod_{v \in S -\{\infty\}} n
\cdot \prod_{v \in S'} \frac{n!}{f_{v,1}!\cdots f_{v,r_v}!}.
\end{equation}
By Theorem~\ref{52} and the relations (\ref{eq:54})--(\ref{eq:59}) , 
we obtain the following result.

\begin{thm}\label{53}
Let $D$ be a definite central simple algebra over $K$
$($with respect to $\infty)$ of degree $n$.
Let $R= R(D/K,\vec{\mathbf{f}})$ be a hereditary $A$-order in $D$ with
invariant $\vec{\mathbf{f}}$. Assume that $n$ is a  
prime number.
Let $S$ be the set of places where $D$ is ramified $($including
$\infty)$, and $S'$ be the set of finite places where the hereditary
$A$-order $R$ is not maximal.
Then the class number $h(D/K,\vec{\mathbf{f}})$ is equal to

\begin{equation}
  \label{eq:510}
  \begin{split}
    & (q-1)\cdot \text{\rm Mass}(D/K,\vec{\mathbf{f}}) \\ 
    + & \frac{q^n-q}{q^n-1}\cdot \frac{\#\text{\rm
        Pic}(O_{L_n})}{n^2}\cdot 
\prod_{v \in S} (n\cdot \epsilon(v)) \cdot \prod_{v \in
  S^{\prime}}\left(\frac{n!}{f_{v,1}!\cdots
    f_{v,r_v}!}(1-\epsilon(v))\right). 
  \end{split}
\end{equation}
\end{thm}

Theorem~\ref{53} agrees with the main results of  Denert and Van Geel
\cite[Theorems 3 and 9]{D-G} when $K$ is the rational function
field. We remark that the 
proof of Theorem~\ref{53} does not rely on the Eichler-Brandt trace
formula. 


\subsection{Example}\label{sec5.3}

Here we present one explicit example to show how to compute the class
number by the recursive formula in \ref{sec4.1}.  

Let $K= \mathbb{F}_q(T)$ with $q = 3$ and $A= \mathbb{F}_q[T]$. Let
$D$ be the central division algebra over $K$ with 
$$-\text{inv}_{\infty}(D) = \text{inv}_{T}(D) = \frac{1}{4} \text{ }
\bmod \mathbb{Z},$$ 
$$\text{inv}_{T+1}(D) = \text{inv}_{T+2}(D) = \frac{1}{2} \text{ }
\bmod \mathbb{Z},$$ 
and $\text{inv}_v(D) = 0 \text{ } \bmod \mathbb{Z}$ for $v \neq
\infty, T, T+1, T+2$. 
Let $R$ be a maximal $A$-order in $D$.
Then the invariant $\vec{\mathbf{f}} = (\vec{f}_v)_{v \neq \infty}$ of
$R$ is: 
\begin{eqnarray}
\vec{f}_T & = & (1) \nonumber \\
\vec{f}_{T+1} & = & (2) \nonumber \\
\vec{f}_{T+2} & = & (2) \nonumber \\
\vec{f}_v & = & (4) \text{ for $v \neq T$, $T+1$, $T+2$.} \nonumber
\end{eqnarray}

By Lemma~\ref{23}, we have $[\mathbb{F}_D:\mathbb{F}_K] = 4$.
It suffices to compute $h_1(D/K,\vec{\mathbf{f}})$,
$h_2(D/K,\vec{\mathbf{f}})$, and $h_4(D/K,\vec{\mathbf{f}})$.\\ 

Note that $L_4 = \mathbb{F}_{q^4}(T)$, $O_{L_4} =
\mathbb{F}_{q^4}[T]$, and 
$D_4' \cong L_4$.
By Theorem~\ref{46},
\begin{eqnarray}
h_4(D/K,\vec{\mathbf{f}}) &=& \frac{q^4-1}{4} \cdot
\text{Mass}(D_4'/L_4) \cdot \prod_{0\leq i \leq 2} \Theta_{T+i}
(D/K,4,\vec{f}_{T+i}) \nonumber \\  
&=& \frac{1}{4} \cdot \prod_{0\leq i \leq 2} \Theta_{T+i}
(D/K,4,\vec{f}_{T+i}). \nonumber  
\end{eqnarray}

Since $\deg T = \deg T+1 = \deg T+2 = 1$, there exists only one place
$w_T$ (resp.\ $w_{T+1}$, $w_{T+2}$) of $L_4$ lying above $T$ 
(resp.\ $T+1$, $T+2$). 
Therefore 
$$\ell_{4,T} = \ell_{4,T+1} = \ell_{4,T+2}=1,$$
and
$$t_{4,T} = 4, t_{4,T+1} = t_{4,T+2} = 2.$$
This also tells us that
$$\Theta_{T+i} (D/K,4,\vec{f}_{T+i}) =
\#(\Omega_{T+i}(D/K,4,\vec{f}_{T+i})).$$ 
When $w = w_T$, for any $\vec{f}_{w,*} = (f_{w,(1,j)})_{1\leq j \leq
  4} \in \Omega_{T}(D/K,4,\vec{f}_T)$, one has 
$$\sum_{j=1}^4 f_{w,(1,j)} = 1.$$
This means that there are $4$ choices for $\vec{f}_{w_T,*}$, i.e.\
$$\#(\Omega_T(D/K,4,\vec{f}_T)) = 4.$$
When $w = w_{T+1} \text{ or } w_{T+2}$, any vector $\vec{f}_{w,*} =
(f_{w,(1,j)})_{1\leq j \leq 2}$ satisfies that 
$$\frac{4}{1\cdot 2} \cdot \sum_{j=1}^2 f_{w,(1,j)} = 2.$$
Therefore there are $2$ choices for $\vec{f}_{w_{T+1},*}$ and
$\vec{f}_{w_{T+2},*}$, i.e.\
$$\#(\Omega_{T+1}(D/K,4,\vec{f}_{T+1})) =
\#(\Omega_{T+2}(D/K,4,\vec{f}_{T+2})) = 2.$$ 
We then conclude that
\begin{equation}
\label{eq51}
h_4(D/K,\vec{\mathbf{f}}) = 4.
\end{equation}
${}$

Next, we compute $h_2(D/K,\vec{\mathbf{f}})$.
It is clear that there exists only one place $w_T$ (resp.\ $w_{T+1}$,
$w_{T+2}$) of $L_2$ lying above $T$ (resp.\ $T+1$, $T+2$), and 
$$\ell_{2,T} = \ell_{2,T+1} = \ell_{2,T+2} = 1, \text{ }
t_{2,T} = t_{2,T+1} = t_{2,T+2} = 2.$$
By Theorem~\ref{46}, we have
$$\frac{h_2(D/K,\vec{\mathbf{f}})}{q^2-1} +
\frac{h_4(D/K,\vec{\mathbf{f}})}{q^4-1} = \frac{1}{2} \cdot
\text{Mass}(D_2'/L_2) \cdot  
\prod_{0\leq i \leq 2} \Theta_{T+i}(D/K,2,\vec{f}_{T+i}).$$ 
The mass formula in Theorem~\ref{31} says that
$$\text{Mass}(D_2'/L_2) = \frac{1}{q^4-1} = \frac{1}{80}.$$
It remains to compute $\Theta_{T+i}(D/K,2,\vec{f}_{T+i})$ for $0\leq i
\leq 2$. 

The vectors $\vec{f}_{w_T,*} = (f_{w_T,(1,j)})_{1\leq j \leq 2} \in
\Omega_{T}(D/K,2,\vec{f}_T)$ 
satisfy that 
$$ \sum_{j=1}^2 f_{w_T,(1,j)} = 1;$$
the vectors $\vec{f}_{w_{T+1},*} = (f_{w_{T+1},(1,j)})_{1\leq j \leq
  2} \in \Omega_{T+1}(D/K,2,\vec{f}_{T+1})$ 
and the vectors $\vec{f}_{w_{T+2},*} = (f_{w_{T+2},(1,j)})_{1\leq j
  \leq 2}\in \Omega_{T+1}(D/K,2,\vec{f}_{T+1})$ 
satisfy that 
$$ \sum_{j=1}^2 f_{w_{T+1},(1,j)} = 2 = \sum_{j=1}^2
f_{w_{T+2},(1,j)}.$$ 
This means that
\begin{eqnarray}
\vec{f}_{w_T,*} &=& (1,0) \text{ or } (0,1), \nonumber \\
\vec{f}_{w_{T+1},*} & = & (2,0) \text{ or } (1,1) \text{ or } (0,2),
\nonumber \\ 
\vec{f}_{w_{T+2},*} & = & (2,0) \text{ or } (1,1) \text{ or }
(0,2). \nonumber 
\end{eqnarray}
Hence
$$\Theta_{T}(D/K,2,\vec{f}_{T}) = 2$$
and
$$\Theta_{T+1}(D/K,2,\vec{f}_{T+1}) =
\Theta_{T+2}(D/K,2,\vec{f}_{T+2}) = 2+\frac{q^4-1}{q^2-1} = 12.$$ 
Therefore we get
\begin{equation}
\label{eq52}
h_2(D/K,\vec{\mathbf{f}}) = 14.
\end{equation}
${}$

Recall
$$\frac{h_1(D/K,\vec{\mathbf{f}})}{q-1} +
\frac{h_2(D/K,\vec{\mathbf{f}})}{q^2-1} +
\frac{h_4(D/K,\vec{\mathbf{f}})}{q^4-1} 
 = \text{Mass}(D/K,\vec{\mathbf{f}}).$$
From the mass formula in Theorem~\ref{31}, we have
$$\text{Mass}(D/K,\vec{\mathbf{f}}) = 169/5.$$
Together with the equations (\ref{eq51}) and (\ref{eq52}), we finally
get 
$$h_1(D/K,\vec{\mathbf{f}}) = 64 \text{ and } h(D/K,\vec{\mathbf{f}})
= 82.$$

\section{Non-principal genera}
\label{sec:06}

In the previous sections, we study the class number $h(R)$ of 
{\it locally free} right ideal classes of $R$ 
(the principal genus). In this section, we 
study the class number of ideal classes of $R$ which are not
necessarily locally free. 
We show that the class number of any genus of 
ideal classes of $R$ can be computed in terms of the that of
locally free ideal classes of another order $R'$, which is still
hereditary.
As a result, our previous results on the computation of the 
class number of the locally free ideal classes of an arbitrary
hereditary order $R$ can be 
extended to that of arbitrary ideal classes of an arbitrary hereditary
order (in $D$). 

Keep the notation as before, in particular $K$, $\infty$, $A$, $D$,
$R$ have the same meaning in the previous sections, 
By a \it right ideal $I$ of $R$ \rm we mean an $A$-lattice $I$ in $D$
which is also a right $R$-submodule of $D$. 
Recall that a genus of right $R$-ideals is a maximal set of right
$R$-ideals 
in $D$ where any two ideals are mutually equivalent locally
everywhere, that is, for two ideals $I$ and $J$,  there exists an
element  
$\alpha_v \in D_v^{\times}$ such that $J_v = \alpha_v \cdot I_v$ for
any finite place $v$ of $K$. Two ideals $I$ and $J$ in a genus $\calL$
are said to be {\it globally equivalent} if there is an element
$\alpha\in 
D^\times$ such that $J=\alpha I$. Let $\calL/\!\sim$ denote the set of
globally equivalent classes in the genus $\calL$ and $h(\calL):=\#
(\calL/\!\sim)$, the cardinality of $\calL/\!\sim$, called the class
number of $\calL$. The set of locally free $R$-ideals forms a genus,
which is called the {\it principal genus}; the others are called
non-principal genera.


Let $v$ be a finite place of $K$.
Suppose the invariant of $R_v$ is 
$\mathbf{f}_v = (f_{v,1},...,f_{v,r_v})$.
Then it is known that (cf.\ \cite{Rei} (39.23) Theorem) any
indecomposable $R_v$-module is one of the form 
$$M_1,M_2,...,M_{r_v}$$
where $M_1 = O_{\Delta_v}^{m_v}$, $M_i = M_{i-1} \cdot
\text{rad}(R_v)$ for $1< i \leq r_v$, and $\text{rad}(R_v)$ is the
Jacobson radical of $R_v$. Recall that $D_v\simeq
\Mat_{m_v}(\Delta_v)$, where $\Delta_v$ is the division part and
$O_{\Delta_v}$ is the maximal order in $\Delta_v$. 
Suppose a right ideal $I$ of $R$ is given.
Then as an $R_v$-module, $I_v$ must be isomorphic to
\begin{equation}
  \label{eq:61}
  M_1^{g_{v,1}^{}} \oplus M_2^{g_{v,2}^{}} \oplus \cdots
\oplus M_{r_v}^{g_{v,r_v}^{}}
\end{equation}
for some $\bfg_v =
(g_{v,1}^{},...,g_{v,r_v}^{})$ with $g_{v,i}^{} \in
\mathbb{Z}_{\geq 0}$ for each $i$ and $\sum_{i=1}^{r_v}
g_{v,i}^{} = m_v$. The datum $\bfg_v^{}$ fixes an
isomorphism class of right $R_v$-ideals, and we call it the
isomorphism type of $I_v$ or the local isomorphism type of $I$ at $v$. 
Therefore we conclude the following result.

\begin{prop}\label{61}
The set of all genera of right $R$-ideals can be parametrized by the
vectors $\bfg:= (\bfg_v)_{v \neq
  \infty}$, where 
\begin{equation}
  \label{eq:62}
  \bfg_v = (g_{v,1}^{},...,g_{v,r_v}^{}) \in
\mathbb{Z}_{\geq 0}^{r_v}\  \text{ with }\  \sum_{i=1}^{r_v}
g_{v,i}^{} = m_v
\end{equation}
and $r_v$ is the period of $R_v$.  
More precisely, let $I({\mathbf{g}}^{})$ be the ideal of $R$
such that for all finite places $v$ 
$$I({\mathbf{g}}^{})_v = M_1^{g_{v,1}^{}} \oplus
M_2^{g_{v,2}^{}} \oplus \cdots \oplus
M_{r_v}^{g_{v,r_v}^{}}.$$ 
Then for any ideal $J$ of $R$, there exists a unique vector
${\mathbf{g}}^{}$ such that $J$ and
$I({\mathbf{g}}^{})$ are in the same genus. 
\end{prop}

\begin{Rem}
When $R_v$ is a maximal order, we have $r_v = 1$ and so $I_v$ must be
isomorphic to $R_v$ as $R_v$-module. Therefore there are only finitely
many genera of $R$. 
\end{Rem}

Let $I$ be a right ideal of $R$, and let $\calL(I)$ be the genus of
right $R$-ideals that contains $I$.
Since any member $J\in \calL(I)$ has the property  $J_v=\alpha_v I_v$,
it follows that 

$$\calL(I) \cong \widehat{D}^{\times} / \widehat{R}_I^{\times}$$
where $R_I$ is the left order of $I$. Therefore,
\begin{equation}
  \label{eq:63}
   \calL(I)/\!\sim\, \cong D^\times \backslash \widehat{D}^{\times} /
\widehat{R}_I^{\times} \cong \Cl(R_I)
\end{equation}
and hence one gets $h(\calL(I))=h(R_I)$, the class number of locally
free right $R$-ideal classes. 

We now describe $R_I$. We may assume that $I=I(\bfg)$ for some vector
$\bfg$ as in (\ref{eq:62}). For each $\bfg_v$ we define another vector
$\bfg_v^{o}$ by removing the zero entries of $\bfg_v$. For example if
$\bfg_v=(3,6,0,1,0)$ (with local period $r_v=5$), then we define the
vector $\bfg^{o}_v$ to be $(3,6,1)$ (with new local period $r_v'=3$).


\begin{prop}\label{62}
For each finite place $v$ of $K$, we have
$$R_{I({\mathbf{g}}^{}),v} =
R_{I({\mathbf{g}}^{})} \otimes_A O_v \cong \text{\rm Mat}(
{\mathbf{g}}_v^{o}, O_{\Delta_v}).$$ 
In particular, $R_{I({\mathbf{g}}^{})}$ is a hereditary $A$-order
in $D$ whose local invariant at $v$ is equal to $\bfg_v^{o}$.
\end{prop}
\begin{proof}
Under the identification of $D_v$ and $\text{Mat}_{m_v}(\Delta_v)$,
$I({\mathbf{g}})_v$ consists of elements $(Y_{i,j})_{1\leq i,j \leq
    r_v}$ satisfying that
$$Y_{i,j} \in \begin{cases}
\text{Mat}_{g_{v,i} \times f_{v,j}} (O_{\Delta_v}) & 
\text{ if $i \leq   j$,} \\
\text{Mat}_{g_{v,i} \times f_{v,j}} (\mathfrak{P}_v) & \text{ if $i > 
  j$.} \end{cases}$$
Let $(Z_{i,j})_{1\leq i,j \leq r_v}$ be an element in $D_v$ where
$Z_{i,j} \in 
\text{Mat}_{g_i\times g_j}(\Delta_v)$. Then 
$(Z_{i,j})_{1\leq i,j \leq r_v}$ is in $R_{I({\mathbf{g}}^{}),v}$ if
and 
only if for any $(Y_{i,j})_{1\leq i,j \leq r_v} \in
I({\mathbf{g}})_v$, 
\begin{equation}
\label{eq:64}
\sum_{k=1}^{r_v} Z_{i,k}\cdot Y_{k,j} \in 
\begin{cases}
\text{Mat}_{g_{v,i} \times f_{v,j}} (O_{\Delta_v}) 
 & \text{ if $i \leq j$,} \\
\text{Mat}_{g_{v,i} \times f_{v,j}} (\mathfrak{P}_v) 
& \text{ if $i >
  j$.} 
\end{cases} 
\end{equation} 
For $1\leq k \leq r_v$,
plugging elements $(Y_{i,j})_{1\leq i,j \leq r_v}$ of
$I({\mathbf{g}})_v$ with $Y_{i,j} = 0$ if $i \neq k$ in
(\ref{eq:64}), 
we get  
$$Z_{i,k}\cdot Y_{k,j} \in 
\begin{cases}
\text{Mat}_{g_{v,i} \times f_{v,j}} (O_{\Delta_v}) 
& \text{ if $i \leq 
  j$,} \\
\text{Mat}_{g_{v,i} \times f_{v,j}} (\mathfrak{P}_v) 
& \text{ if $i > 
  j$.} \end{cases}$$ 
This implies that $(Z_{i,j})_{1\leq i,j \leq r_v}$ is in
$R_{I(\bfg),v}$ if and only if  
$$Z_{i,j} \in \begin{cases}
\text{Mat}_{g_{v,i} \times g_{v,j}} (O_{\Delta_v}) 
& \text{ if $i \leq
  j$,} \\
\text{Mat}_{g_{v,i} \times g_{v,j}} (\mathfrak{P}_v) 
& \text{ if $i >
  j$.} \end{cases}$$
Therefore $R_{I(\bfg),v}$ is a hereditary $O_v$-order with
invariant $\mathbf{g}_v^{o}$.
\qed
\end{proof}


Together with the parametrization of the genera of $R$-ideals in
Proposition~\ref{61}, we obtain the following result. 

\begin{thm}\label{63}
The total class number of right ideal classes of $R$ is equal to
$$\sum_{{\mathbf{g}}^{}} h({\mathbf{g}}^{}),$$
where $\bfg$ runs
through the vectors described in Proposition~\ref{61} and
$h({\mathbf{g}}^{}):=h(R_{I({\mathbf{g}}^{})})$.
\end{thm}

\begin{Rem}
After computing the class
number of $R_{I({\mathbf{g}}^{})}$ for each $\mathbf{g}$, 
we obtain the total class number of right ideal classes of $R$. 
\end{Rem}

\section{Local optimal embeddings}\label{sec:07}

\subsection{}
Fix a non-Archimedean local field $F$.
Let $D$ be a finite dimensional central simple algebra over $F$.
Take any simple left $D$-submodule $V$ of $D$. 
Let $\Delta = \big(\text{End}_D(V)\big)^{\text{op}}$, which is a
central division algebra over $F$. 
Then $V$ can be viewed as a free right $\Delta$-module, and
$D$ is canonically isomorphic to $\text{End}_{\Delta}(V)$.
Let $m := \text{rank}_{\Delta}(V)$. Then
$$D \cong \text{Mat}_m(\Delta).$$
Let $d$ be the positive integer such that $d^2 = [\Delta : F]$.
For any field extension $F'$ over $F$, there exists an embedding of
$F'$ into $D$ if and only if $[F':F]$ divides $n:= md$. 
In general, let
$$L = L_1\times L_2 \times \cdots \times L_{\ell}$$
where for $1\leq w \leq \ell$, $L_w$ is a finite field extension 
of $F$. Then

\begin{lem}\label{71} 
$(1)$ There exists an $F$-algebra embedding $\iota: L \hookrightarrow
D \cong \text{\rm End}_{\Delta}(V)$ if and only if there exists
positive integers $m_1,...,m_{\ell}$ such that 
$$m_1+\cdots +m_{\ell} = m = \text{\rm rank}_{\Delta}(V) \quad
\text{and}\quad [L_w:F] \mid m_w d.$$ 
$(2)$ Given two embeddings $\iota_1$ and $\iota_2$ of $L$ into $D$.
There exists an element $g$ in $D^{\times}$ such that
$$\iota_1(\alpha) = g^{-1} \iota_2(\alpha) g, \text{ } \forall \alpha
\in L$$ 
if and only if $\text{\rm rank}_{\Delta}(\iota_1(e_w)V) = \text{\rm
  rank}_{\Delta}(\iota_2(e_w)V)$ for $1\leq w \leq \ell$, 
where $e_w$ is the idempotent in $L$ corresponding to $L_w$.
\end{lem}

\begin{proof}
Suppose there exists an $F$-algebra embedding $\iota: L
\hookrightarrow \text{End}_{\Delta}(V)$. 
For $1\leq w \leq \ell$, let $e_w$ be the idempotent in $L$
corresponding to $L_w$. 
Then $\iota(e_w)V \neq 0$ and $\iota$ induces an embedding of $L_w$
into $\text{End}_{\Delta}(\iota(e_w)V)$. 
Set $m_w:= \text{rank}_{\Delta}(\iota(e_w)V)$ for each $w$.
We have
$$m_1+\cdots +m_{\ell} = m \text{ }\text{ and } \text{ }[L_w:F] \mid
m_w d.$$ 

Conversely, suppose we can find positive integers $m_1,...,m_{\ell}$
satisfying the desired property. 
Consider any decomposition
$$V = V_1 \oplus V_2 \oplus \cdots \oplus V_{\ell}$$ 
with $\text{rank}_{\Delta}(V_w) = m_w$.
There exists $\iota_w : L_w \hookrightarrow \text{End}_{\Delta}(V_w)$
for $1\leq w \leq \ell$. 
Then the composition of $\iota_1\times \cdots \times \iota_{\ell}$ and
the natural embedding of $\prod_{w=1}^{\ell}
\text{End}_{\Delta}(V_w)$ into $\text{End}_{\Delta}(V)$ gives an
embedding $\iota: L \hookrightarrow D$. 
Therefore the proof of $(1)$ is complete.  

For $(2)$, if there exists $g \in D^{\times}$ such that
$$\iota_1(\alpha) = g^{-1} \iota_2(\alpha) g,$$
then for $1\leq w \leq \ell$,
$$\text{rank}_{\Delta}(\iota_1(e_w)V) =
\text{rank}_{\Delta}(g^{-1}\iota_2(e_w)V) =
\text{rank}_{\Delta}(\iota_2(e_w)V).$$ 
Conversely, suppose $\text{\rm rank}_{\Delta}(\iota_1(e_w)V) =
\text{\rm rank}_{\Delta}(\iota_2(e_w)V)$ for $1\leq w \leq \ell$. 
Then $\iota_1(e_w)V$ and $\iota_2(e_w)V$ are isomorphic as
$(L_w,\Delta)$-bimodules since $\Delta \otimes_F L_w$ is a central
simple algebra over $L_w$. 
Let $g_w: \iota_1(e_w)V \isoto \iota_2(e_w)V$ be an
isomorphism of $(L_w,\Delta)$-bimodules. 
Then the element $g$ in $D^{\times}$ corresponding to the automorphism
$$g_1\times \cdots \times g_{\ell}: V = \bigoplus_{w=1}^{\ell}
\iota_1(e_w)V \rightarrow \bigoplus_{w=1}^{\ell} \iota_2(e_w)V = V$$ 
satisfies that
$$\iota_1(\alpha) = g^{-1} \iota_2(\alpha) g, \text{ } \forall \alpha
\in L.$$ 
This completes the proof of $(2)$. \qed
\end{proof}

The vector $(m_1,...,m_{\ell})$ where $m_w =
\text{rank}_{\Delta}(\iota(e_w)V)$ is called the \it type of the
embedding $\iota: L \hookrightarrow D$. \rm  

Now, we fix an $F$-algebra embedding $\iota:L = L_1\times \cdots
\times L_{\ell} \hookrightarrow D$ of type $(m_1,...,m_{\ell})$. 
Denote by $O_F$ (resp.\ $O_{L_w}$) the valuation ring of $F$ 
(resp.\ $L_w$) 
and set
$$O_L:= O_{L_1}\times \cdots \times O_{L_{\ell}}.$$
For any $O_F$-order $R$ of $D$, we call $\iota$ \it an optimal
embedding of $O_L$ into $R$ \rm if  
$$\iota(L) \cap R = \iota(O_L).$$
Two optimal embeddings $\iota_1$ and $\iota_2$ of the same type are
called \it equivalent modulo $R^{\times}$ \rm 
if there exists an element $u \in R^{\times}$ such that 
$$\iota_1(\alpha) = u \cdot \iota_2(\alpha) \cdot u^{-1}, \text{ }
\forall \alpha \in L.$$ 
Then Lemma \ref{71} implies that

\begin{lem}\label{72}
The set of equivalence classes of optimal embeddings of type
$(m_1,...m_{\ell})$ from $O_L$ into $R$ modulo $R^{\times}$ can be
identified with 
$$C_{\iota}(L)^{\times} \backslash \mathcal{E}_{\iota}(R) / R^{\times}
$$ 
where $C_{\iota}(L)$ is the centralizers of $\iota(L)$ in $D$ and
\begin{equation}
  \label{eq:71}
  \mathcal{E}_{\iota}(R):= \{ g \in D^{\times} : \iota(L)\cap g R
g^{-1} = \iota(O_L)\}.
\end{equation}
\end{lem}

Suppose $R$ is \it hereditary, \rm i.e.\ any right (or equivalently
left) ideal of $R$ is projective as an $R$-module. Let $O_{\Delta}$ be 
the maximal compact subring of $\Delta$ and $\mathfrak{P}_{\Delta}$ be
the maximal two-sided ideal of $O_{\Delta}$. Then there exists a
vector $\vec{f}_R = (f_1,...,f_r) \in \mathbb{Z}_{> 0}^r$ such that 
$$R \cong \text{Mat}_m(\vec{f}_R, O_{\Delta}),$$
where the ring $\text{Mat}_m(\vec{f}_R, O_{\Delta})$ 
consists of elements $(X_{i,j})_{1\leq i, j \leq r}$ in
$\text{Mat}_{m}(O_{\Delta})$ such that
$$X_{i,j} \in 
\begin{cases}
\text{Mat}_{f_{i} \times f_{j}} (O_{\Delta}) 
  & \text{ if $i \leq j$,} \\
\text{Mat}_{f_{i} \times f_{j}} (\mathfrak{P}_{\Delta}) 
  & \text{ if $i > j$.} 
\end{cases}$$
The number $r$ is called the {\it period} of $R$;  
the vector $\vec{f}_R:= (f_{1},...,f_{r})$ is called the 
{\it invariant of $R$}, which is uniquely determined by $R$ 
up to cyclic permutations. 
When $r = m$ and $f_i = 1$ for all $i$, the order $R$ is 
an Iwahori order, which is conjugate to 
the preimage of the set of upper triangular
matrices over 
$\mathbb{F}_{\Delta} := O_{\Delta}/\mathfrak{P}_{\Delta}$.
Next, we  connect optimal embeddings of a given type from $O_L$ into a
hereditary $O_F$-order $R$ of $D$ modulo $R^{\times}$ with the \lq\lq
flags\rq\rq\ of $(O_L,O_{\Delta})$-bimodules. 

\subsection{$(O_L,O_{\Delta})$-flags}\label{sec1.1}

Recall that $D \cong \text{End}_{\Delta}(V)$ and $V$ is a free right
$\Delta$-module of $\text{rank}_{\Delta}(V) = m$.  

\begin{defn}\label{73}
Let $N_1,...,N_{r}$ be free right $O_{\Delta}$-modules in $V$ of
$O_{\Delta}$-rank $m$. 
We call $N_*:= (N_1,...,N_{r})$ \it an $O_{\Delta}$-flag in $V$ of type
$\vec{f} = (f_1,...,f_r) \in \mathbb{Z}_{> 0 }^r$ \rm if 
\begin{itemize}
\item[(1)] $N_1 \supsetneq N_2 \supsetneq \cdots \supsetneq N_{r}
  \supsetneq N_{r+1}:=N_1 \mathfrak{P}_{\Delta}$. 
\item[(2)] $\text{dim}_{\mathbb{F}_{\Delta}} (N_{i}/N_{i+1})=f_{i}$
  for $1\leq i \leq r$. 
\end{itemize}
The set $\text{Hom}_{O_{\Delta}}(N_*, N_*')$ of morphisms of
$O_{\Delta}$-flags in $V$ of type $\vec{f}$ consists of endomorphisms $\phi
\in \text{End}_{\Delta}(V)$ such that $\phi(N_i) \subseteq N_i'$ for
$1\leq i \leq r$. 
\end{defn}

Fix a $\Delta$-basis $\{x_1,...,x_m\}$ of $V$.
Given a vector $\vec{f} = (f_{1},...,f_{r}) \in \mathbb{Z}_{>0}^r$,
set 
$$f_{(i)}:= \sum_{j = 1}^i f_{j} \quad \text{for } 1\leq i 
\leq r, \text{ and } f_{(0)}:= 0.$$
Then we have a strictly increasing sequence of integers 
$0 = f_{(0)} < f_{(1)}<\cdots < f_{(r)} = m$. 
For $1 \leq i \leq r$, let 
$$M_i:= \left( \bigoplus_{k = 1}^{m-f_{(i-1)}} x_{k}
  O_{\Delta}\right) \oplus 
             \left( \bigoplus_{k = m-f_{(i-1)}+1}^{m} x_{k}
               \mathfrak{P}_{\Delta}\right).$$ 
Then $M_*(\vec{f}):= (M_1,...,M_r)$ is an $O_{\Delta}$-flag in $V$ of type
$\vec{f}$, and 
$$\text{End}_{O_{\Delta}}(M_*(\vec{f})) \cong
\text{Mat}_{m}(\vec{f},O_{\Delta}).$$ 
It is clear that every $O_{\Delta}$-flag in $V$ 
of type $\vec{f}$ is of the form 
$$g M_*(\vec{f}):=(g M_1,...,g M_{r})$$ 
where $g$ is in $\text{Aut}_{\Delta}(V) \cong D^{\times}$, and 
$$\text{End}_{O_{\Delta}}(gM_*(\vec{f})) = g
\text{End}_{O_{\Delta}}(M_*(\vec{f})) g^{-1}.$$ 
${}$

Suppose an $F$-algebra embedding $\iota:L \hookrightarrow D$ of type
$(m_1,...,m_{\ell})$ is given.  

\begin{defn}\label{74}
Let $N_* = (N_1,...,N_r)$ be an $O_{\Delta}$-flag in $V$ of type $\vec{f}$. 
We call $N_*$ an \it $(O_{L},O_{\Delta})$-flag in $V$ of type
$(\iota,\vec{f})$ \rm if 
$$\iota(O_{L})\cdot N_i \subseteq N_i \text{ for all $1\leq i \leq
  r$.}$$ 
The set $\text{Hom}_{(O_{L},O_{\Delta})}(N_*, N_*')$ of morphisms of
$(O_{L},O_{\Delta})$-flags in $V$ of type $(\iota,\vec{f})$ consists of
morphisms in $\text{Hom}_{O_{\Delta}}(N_*, N_*')$ which commute with
$\iota(O_{L})$. 
\end{defn}

Let $R$ be a hereditary $O_F$-order of $D$ with invariant
$\vec{f}_R$. This means that there exists an isomorphism from $D$ to
$\text{Mat}_m(\Delta)$ that sends $R$ onto the order
$\text{Mat}_m(\vec{f}_R,O_{\Delta})$.  
Given $g \in D^{\times}$, suppose $g M_*(\vec{f}_R)$ is an
$(O_{L},O_{\Delta})$-flag in $V$ of type $(\iota,\vec{f}_R)$. 
Then we must have
$$g^{-1} \iota(O_{L}) g \subset
\text{End}_{O_{\Delta}}(M_*(\vec{f}_R)) = R,$$ 
which means that $g$ is in $\mathcal{E}_{\iota}(R)$ (see
(\ref{eq:71})).  
Two $(O_{L},O_{\Delta})$-flags $g_1M_*(\vec{f}_R)$ and
$g_{2}M_*(\vec{f}_R)$ in $V$ of type $(\iota,\vec{f}_R)$ are isomorphic if
and only if there exists an element $h$ in $C_{\iota}(L)^{\times}$
such that 
$g_{2}^{-1} h g_{1} \in R^{\times}$.
If we denote by $\text{\bf FL}(O_{L},O_{\Delta},m,\iota,\vec{f})$ the
set of isomorphism classes of $(O_{L},O_{\Delta})$-flags of a given
type $(\iota,\vec{f})$ in $V$, then we conclude that 

\begin{prop}\label{75}
We have the following bijection
\begin{equation}
  \label{eq:72}
  \begin{tabular}{ccc}
$C_{\iota}(L)^{\times} \backslash \mathcal{E}_{\iota}(R) / R^{\times}$
& $\longleftrightarrow$ & $\text{\bf
  FL}(O_{L},O_{\Delta},m,\iota,\vec{f}_R)$ \\ 
$C_{\iota}(L)^{\times} g R^{\times}$ & $\longmapsto$ &
$[gM_*(\vec{f}_R)]$. 
\end{tabular}
\end{equation}

Moreover, let $g M_*(\vec{f}_R)$ be an 
$(O_{L},O_{\Delta})$-flag in $V$ of type $(\iota,\vec{f}_R)$.
One has
$$ C_{\iota}(L) \cap g R g^{-1} \cong \text{\rm
  End}_{(O_{L},O_{\Delta})}(g M_*(\vec{f}_R))$$ 
and
$$C_{\iota}(L)^{\times} \cap g R^{\times} g^{-1} \cong \text{\rm
  Aut}_{(O_{L},O_{\Delta})}(g M_*(\vec{f}_R)).$$ 
\end{prop}

\begin{Rem}
For any two embeddings $\iota_1$ and $\iota_2$ of $L$ into $D$, the
sets $\text{\bf FL}(O_{L},O_{\Delta},m,\iota_1,\vec{f}_R)$
and $\text{\bf FL}(O_{L},O_{\Delta},m,\iota_2,\vec{f}_R)$ are
isomorphic if $\iota_1$ and $\iota_2$ have the same type. 
\end{Rem}

\subsection{Decomposition of $\text{\bf
    FL}(O_{L},O_{\Delta},m,\iota,\vec{f})$} 

In this subsection, we fix an $F$-algebra embedding $\iota$ of $L =
L_1\times \cdots \times L_{\ell}$ into $D$ of type
$(m_1,...,m_{\ell})$. 
The centralizer $C_{\iota}(L)$ of $\iota(L)$ in $D$ is canonically
isomorphic to 
$$\prod_{w=1}^{\ell} \text{End}_{(L_w,\Delta)}(\iota(e_w)V),$$
where $e_w$ is the idempotent in $L$ corresponding to $L_w$,
and $\text{End}_{(L_w,\Delta)}(\iota(e_w)V)$ is the endomorphism ring
of the $(L_w,\Delta)$-bimodule $\iota(e_w)V$, which is a central
simple algebra over $L_w$.  

Let $N_* = (N_1,...,N_r)$ be an $(O_L,O_{\Delta})$-flag in $V$ of type
$(\iota,\vec{f})$. 
For each $w$ with $1\leq w \leq \ell$, we have the chain of lattices
$$\iota(e_w)N_1 \supseteq \cdots \iota(e_w)N_r \supseteq
\iota(e_w)N_{r+1} = \iota(e_w)N_1\mathfrak{P}_{\Delta} \text{ in
  $V$}.$$ 
Put $f_{w,i}:=
\text{dim}_{\mathbb{F}_{\Delta}}(\iota(e_w)N_i/\iota(e_w)N_{i+1})$.
From $N_i=\oplus_w \iota(e_w)N_i$ it
follows that
\begin{equation}
  \label{eq:73}
  \sum_{w=1}^{\ell}f_{w,i} = f_i\quad \text{and}
  \quad \sum_{i=1}^r f_{w,i}=m_w. 
\end{equation}
We denote by $\vec{f}_w^o$ the vector obtained by removing zero
entries of 
$\vec{f}_w=(f_{w,1},...,f_{w,r})$, i.e. \ 
$$\vec{f}_w^o = (f_{w,j_1},...,f_{w,j_{r_w}})$$
where $1\leq j_1 <\cdots < j_{r_w} \leq r$, $f_{w,j_k}>0$, and
$f_{w,j} = 0$ for $j \neq j_1,...,j_{r_w}$. 
Then
\begin{equation}
  \label{eq:74}
  \iota(e_w)N_*:= (\iota(e_w)N_{j_1},...,\iota(e_w)N_{j_{r_w}})
\end{equation}
is an $(O_{L_w},O_{\Delta})$-flag in $\iota(e_w)V$ of type
$(\iota_w,\vec{f}_w^o)$, 
where $\iota_w$ is the restriction of $\iota$ on $O_{L_w}$. 

Conversely, suppose we have an $(O_{L_w},O_{\Delta})$-flag $N_{i,*}
=(N_{w,1},...,N_{w,r_w})$ in $\iota(e_w)V$ of type
$(\iota,\vec{f}_w^o)$ for $1\leq w \leq \ell$, where 
$\vec{f}_w^o$ is obtained by removing zero entries of a vector
$\vec{f}_w = (f_{w,1},...,f_{w,r})$, with 
$$\sum_{w=1}^{\ell} f_{w,i} = f_i \quad \text{and}
  \quad \sum_{i=1}^r f_{w,i}=m_w $$
Let 
$$N_i= \oplus_{w=1}^{\ell} N_{w,i}\subset V, \quad  1\leq i \leq r.$$
Then
$N_* = (N_1,...,N_r)$ is an $(O_L,O_{\Delta})$-flag in $V$ of type
$(\iota,\vec{f})$.  

It is clear that two $(O_L,O_{\Delta})$-flags $N_*$ and $N_*^{\prime}$
in $V$ are isomorphic if and only if $\iota(e_w)N_*$ and
$\iota(e_w)N_*^{\prime}$ are isomorphic $(O_{L_w},O_{\Delta})$-flags
in $\iota(e_w)V$ for all $w=1,\dots, \ell$. We get the following result.  

\begin{prop}\label{76}
{\rm (1)} One has
\begin{equation}
  \label{eq:75}
  \text{\bf FL}(O_L,O_{\Delta},m,\iota,\vec{f})
= \coprod_{(\vec{f}_1,...,\vec{f}_{\ell})} 
\left(\prod_{w=1}^{\ell} \text{\bf FL}(O_{L_w},O_{\Delta},m_w, \iota_w,
  \vec{f}_w^o)\right),
\end{equation}
where $(\vec{f}_1,...,\vec{f}_{\ell})$ runs through all tuples of
vectors $\vec{f}_w=(f_{w,i})_i \in
  \mathbb{Z}_{\geq 0}^r$ of non-negative integers satisfying the
  condition (\ref{eq:73}), and 
$\vec{f}_w^o$ is the vector obtained by removing zero entries of 
the vector $\vec{f}_w$.

{\rm (2)} Given an $(O_L,O_{\Delta})$-flag $N_*$ in $V$ of type
$(\iota,\vec{f})$, we have 
$$\text{\rm End}_{(O_L,O_{\Delta})}(N_*) = \prod_{w=1}^{\ell}
\text{\rm End}_{(O_{L_w},O_{\Delta})}(\iota(e_w)N_*).$$ 
and
$$\text{\rm Aut}_{(O_L,O_{\Delta})}(N_*) = \prod_{w=1}^{\ell}
\text{\rm Aut}_{(O_{L_w},O_{\Delta})}(\iota(e_w)N_*).$$ 
\end{prop}

\begin{Rem}
The above proposition tells us that to understand $\text{\bf
  FL}(O_L,O_{\Delta},m,\iota,\vec{f})$, 
it suffices to focus on the case when $L$ is a field.
\end{Rem}

\subsection{Special case: when $L$ is an unramified field
  extension}\label{sec1.3} 

Fix an embedding $\iota: L \hookrightarrow D$ as usual.
In this subsection, we describe explicitly the isomorphism classes in
$\text{\bf FL}(O_{L},O_{\Delta},m,\iota,\vec{f})$ for the case where
$L$ is an unramified field extension over $F$. 
Recall the following basic result: 

\begin{prop}\label{77} 
Let $L$ be a finite extension over $F$.
If we let $$t:= \text{\rm gcd}([L:F], d),$$
where $d^2=[\Delta:F]$, then
the central simple algebra $\Delta \otimes_F L$ over $L$ is isomorphic
to $\text{\rm Mat}_t(\Delta')$ 
where  $\Delta'$ is a central division algebra over $L$ with 
$\text{dim}_{L} \Delta' = (d/t)^2$.
\end{prop}

From now on, $L$ is an unramified field extension of $K$. 
Choose an unramified maximal subfield $W$ in $\Delta$.
Then $[W:F] = d$.
Let $\pi$ be a uniformizer of $F$.
There exists  an element $u$ in $\Delta$ and a positive integer
$\kappa$ with $\text{gcd}(\kappa,d) = 1$ such that  
$$u^{d} = \pi^{\kappa} \text{ and }
u \alpha = \text{Frob}_{W/F}(\alpha) u, \text{ } \forall\, \alpha \in
W.$$ 
Here $\text{Frob}_{W/F}$ is the Frobenius automorphism of $W$ over
$F$. 
Therefore $\Delta$ is isomorphic to the cyclic algebra
$(W/F,\text{Frob}_{W/F}, \pi^{\kappa})$ 
(cf.\ \cite{Rei} \S 30 and \S 31). 
The integer $\kappa \bmod d$ is independent of the choices of $W$ and 
$u$, 
and
$$\frac{\kappa}{d} \bmod \mathbb{Z} \in \mathbb{Q}/\mathbb{Z}$$
is called the \it Hasse invariant of \rm $D = \text{Mat}_m(\Delta)$.

It is known that $\Delta'$ is isomorphic to the cyclic algebra
$$(WL/L,\text{Frob}_{WL/L}, \pi^{\kappa'})$$ 
where $\kappa' \equiv \kappa \cdot ([L:F]/t) \bmod (d/t)$ (cf.\
\cite{Rei} (31.9)). 

Let $\gamma$ (resp.\ $\gamma'$) be a positive integer such that 
$$\gamma \cdot \kappa \equiv 1 \bmod d, \quad \text{ (resp.\
$\gamma' \cdot \kappa' \equiv 1 \bmod (d/t)$)}.$$
There exists an element $\Pi$ (resp.\ $\Pi'$) in $\Delta$ 
(resp.\ $\Delta'$) such that $\Pi^d = \pi$ 
(resp.\ $(\Pi')^{d/t} = \pi$) and 
$$\Pi \alpha = \text{Frob}_{W/F}^{\gamma} (\alpha) \Pi, \quad \text{ }
\forall\, \alpha \in W$$ 
$$\text{(resp.\ $\Pi' \alpha' = 
\text{Frob}_{WL/L}^{\gamma'} (\alpha')
  \Pi', \quad \forall\, \alpha' \in WL$)}.$$ 
Obviously, 
$$\Pi' \alpha = \text{Frob}_{W/F}^{\gamma \cdot t} (\alpha) \Pi'
\quad \forall\, \alpha \in W.$$ 
Now, the isomorphism between $\Delta \otimes_{F} L$ and
$\text{Mat}_{t}(\Delta')$ can be described by the following: 
$$\begin{tabular}{ccll}
$\Pi \otimes 1$ & $\longmapsto$ & 
$\begin{pmatrix} 0 & 1 &  & \\  & 
  \ddots & \ddots & \\  & & 0 & 1 \\ \Pi' & & & 0 \end{pmatrix}$, 
& \\  
$\alpha \otimes 1$ & $\longmapsto$ & 
$\begin{pmatrix} \alpha & & & \\
  & \text{Frob}_{W/F}^{\gamma}(\alpha) & & \\ 
 & & \ddots & \\ & & & \text{Frob}_{W/F}^{\gamma (t-1)}(\alpha)
\end{pmatrix}$ 
& $\forall\, \alpha \in W$,\\ 
$1 \otimes \beta$ & $\longmapsto$ & $\begin{pmatrix} \beta & & \\ &
  \ddots & \\ & & \beta \end{pmatrix}$ & $\forall\, \beta \in L$. 
\end{tabular}$$

In particular, denote by $O_{\Delta'}$ the maximal compact subring in 
$\Delta'$ and  
$\mathfrak{P}_{\Delta'} := \Pi' O_{\Delta'}$, the maximal two-sided
ideal in $O_{\Delta'}$. 
Then
$$O_{\Delta} \otimes_{O_F} O_{L} \cong
\begin{pmatrix} O_{\Delta'} & O_{\Delta'} & \cdots & O_{\Delta'} \\
\mathfrak{P}_{\Delta'} & O_{\Delta'} & \cdots & O_{\Delta'} \\
\vdots & \ddots & \ddots & \vdots \\
\mathfrak{P}_{\Delta'} & \cdots & \mathfrak{P}_{\Delta}' & O_{\Delta'}
\end{pmatrix} \subset \text{Mat}_{t}(\Delta').$$ 
We conclude that
\begin{lem}\label{78}
Suppose $L$ is unramified over $F$. Then $O_{\Delta} \otimes_{O_F}
O_L$ is isomorphic to an Iwahori order in $\text{\rm
  Mat}_t(\Delta')$. 
\end{lem}

Recall that $\iota: L \hookrightarrow D = \text{End}_{\Delta}(V)$ 
is a fixed embedding. 
Then $V$ can be viewed as an $(L,\Delta)$-bimodule via $\iota$.
It is natural to think $V$ as a right $\Delta \otimes_F L$-module.
For $1\leq i,j \leq t$, let $E_{i,j}$ be the $(i,j)$-th elementary
matrix in $\text{Mat}_{t}(\Delta') \cong \Delta \otimes_{F} L$ 
(i.e.\ whose $(i,j)$-entry is one and other entries are zero). 
Let 
$$V^{(i)}:= V E_{i,i} \quad \text{for } 1\leq i \leq t.$$
The multiplication by $E_{i,j}$ from the right gives rise to an
isomorphism of free right $\Delta'$-modules 
$$ \begin{tabular}{cccl}
$\varphi_{j,i}:$ & $V^{(i)}$ & $\longrightarrow$ & $V^{(j)}$\\
  & $xE_{i,i}$ & $\longmapsto$ & $xE_{i,j} = \big((xE_{i,i}) \cdot
  E_{i,j}\big) \cdot E_{j,j}$. 
\end{tabular}$$
Therefore 
$$m':=\text{rank}_{\Delta'}V^{(i)} = \frac{1}{t} \cdot
\text{rank}_{\Delta'} V 
    = \frac{m[\Delta:F]}{t [\Delta':F]}=\frac{m \cdot t}{[L:F]}$$
as $t^2[\Delta':F]=[\Delta\otimes_F L:F]=[\Delta:F][L:F]$. 

\begin{lem}\label{79}
For $1\leq i \leq t$, we have
$$\begin{tabular}{ccc}
$\text{\rm End}_{\Delta \otimes_{F} L} (V)$ & $\cong$ & $\text{\rm
  End}_{\Delta'}(V^{(i)})$ \\ 
$g$ & $\mapsto$ & $g \mid_{V^{(i)}}$.
\end{tabular}$$
\end{lem}

\begin{proof}
For a given $h \in \text{End}_{\Delta'}(V^{(i)})$,
consider the following map 
$$\Phi_h:= \sum_{j=1}^{t} \varphi_{j,i} \circ h 
\circ \varphi_{i,j} :
\bigoplus_{j=1}^{t} V^{(j)} \longrightarrow \bigoplus_{j=1}^{t}
V^{(j)}.$$ 
It is clear that $\Phi_h \in \text{\rm End}_{\Delta \otimes_{F} L}
(V)$ and $\Phi_h |_{V^{(i)}} = h$. 

On the other hand, let $g \in \text{\rm End}_{\Delta \otimes_{F} L}
(V)$. 
For any $\alpha$ in $V$,
$$g(\alpha) = g( \sum_{j=1}^{t} \alpha E_{j,j})
= \sum_{j=1}^{t} \big(g(\alpha E_{j,i})\big) E_{i,j} = 
\sum_{j=1}^{t} (\varphi_{j,i} \circ (g \mid_{V^{(i)}}) 
\circ \varphi_{i,j})(\alpha)$$ 
and the result follows. \qed
\end{proof}

For $1 \leq i < t$, one has 
$$\Pi^i \cdot E_{t,t} = E_{t-i,t} \text{ and } 
\Pi^{t} \cdot E_{t,t} = 
E_{t,t} \cdot \Pi'.$$ 
Let $N_{*} = (N_{1},...,N_{r})$ be an 
$(O_{L},O_{\Delta})$-flag in $V$ 
of type $\vec{f} = (f_{1},...,f_{r})$. 
Then we can view $N_*$ as a flag of right
$O_{\Delta}\otimes_{O_F}O_L$-modules (in $V$). 
Set $$N_{i,j}:= N_{i} \cdot E_{t+1-j,t} \subset V^{(t)}\quad \text{for\ }
1\leq i \leq r \text{ and } 1\leq j \leq t.$$ 
Then multiplying $E_{t,t}$ (from the right) on the chain
$$ N_{1} \supsetneq N_{2} \supsetneq \cdots \supsetneq N_{r}
\supsetneq N_{r+1} = N_{1} \cdot \Pi,$$ 
we get a longer chain of right $O_{\Delta'}$-lattices in $V^{(t)}$: 
\begin{equation}
  \label{eq:76}
  \begin{tabular}{llllllll}
 & $N_{1,1}$ & $\supseteq$ & $N_{2,1}$ & $\supseteq$ & $\cdots$ &
 $\supseteq$ & $N_{r,1}$ \\ 
$\supseteq$ & $N_{1,2}$ & $\supseteq$ & $N_{2,2}$ & $\supseteq$ &
$\cdots$ & $\supseteq$ & $N_{r,2}$ \\ 
$\supseteq$ & $\cdots$ & & $\cdots$ & & $\cdots$ &  & $\cdots$\\
$\supseteq$ & $N_{1,t}$ & $\supseteq$ & $N_{2,t}$ & $\supseteq$ &
$\cdots$ & $\supseteq$ & $N_{r,t}$ \\ 
$\supseteq$ & $N_{1,1} \cdot \Pi'$. &&&&&&
\end{tabular}
\end{equation}
Moreover, 
$$N_{i} = \bigoplus\limits_{j=1}\limits^{t}
N_{i,j}E_{t,t+1-j} \quad \text{for } 1\leq i \leq r.$$ 
Recall that $\mathbb{F}_{\Delta} = O_{\Delta}/\mathfrak{P}_{\Delta}$
and we let $\mathbb{F}_{\Delta'} =
O_{\Delta'}/\mathfrak{P}_{\Delta'}$. 
Let 
\begin{equation}
  \label{eq:77}
  f_{i,j}:= 
\begin{cases}
\text{dim}_{\mathbb{F}_{\Delta'}}(N_{i,j}/N_{i+1,j}) 
  & \text{ for  $1\leq i <r$},\\ 
\text{dim}_{\mathbb{F}_{\Delta'}}(N_{r,j}/N_{1,j+1}) 
  & \text{ for $i = r$ and $1\leq j < t$},\\ 
\text{dim}_{\mathbb{F}_{\Delta'}}\big(N_{r,t}/(N_{1,1} \Pi')\big) 
  & \text{ for $i = r$ and $j=t$}.\end{cases}
\end{equation}
Then for $1\leq i \leq r$,
$$ [\mathbb{F}_{\Delta'} : \mathbb{F}_{\Delta}]\cdot  
\sum_{j=1}^{t} f_{i,j} 
 = \frac{[L:F]}{t} \cdot \sum_{j=1}^{t} f_{i,j} = f_{i}.$$
This tells us that $[L:F]/t$ divides $f_{i,j}$ for all $i$.  

We fix an order of the index set $\{(i,j) \}_{1\leq i \leq r, 1\leq j
  \leq t}$: 
$$(i,j) < (i',j') \text{ if } 
\begin{cases} j < j' & \text{} \\ {j =
    j' \text{ and } i < i'.} & \text{ } \end{cases}$$ 
Consider the vector $$\vec{f}_* = (f_{i,j})_{1\leq i \leq r, 1\leq j
  \leq t}  
= (f_{1,1}, ...,
f_{r,1},f_{1,2},...,f_{r,2},...,f_{1,t},...,f_{r,t}).$$ 
Let $\vec{f}_*^{o}$ be the vector obtained by removing the zero
entries of the vector $\vec{f}_*$. 
Picking the \lq\lq gap\rq\rq\ of the chain 
$(N_{i,j})_{1\leq i \leq r,
  1\leq j \leq t}$ as (\ref{eq:74}),
we obtain an
$O_{\Delta'}$-flag in $V^{(t)}$ of type $\vec{f}_{*}^{o}$.
It is clear that 
$$\text{End}_{E(O_{L},O_{\Delta})}(N_{*}) \cong
\text{Mat}_{m'}(\vec{f}_*^{o},O_{\Delta'}) \subset
\text{Mat}_{m'}(\Delta').$$ 

Conversely, suppose $[L:F]/t$ divides $f_{i}$ for all $i$. 
Then for any such a chain of $O_{\Delta'}$-lattices
$\big(N_{i,j}\big)_{1 \leq i \leq r, 1\leq j \leq t}$ in $V^{(t)}$, 
we can rebuild an $(O_{L},O_{\Delta})$-flag $N_* = (N_1,...,N_r)$ in
$V$ of type $(\iota,\vec{f})$ by setting 
$$N_i = \bigoplus_{j=1}^t N_{i,j}\cdot E_{t,t+1-j} \subset
\bigoplus_{j=1}^t V^{(t+1-j)} = V.$$ 
Two $(O_{L},O_{\Delta})$-flags $N_{*}$ and $N_{*}'$ are isomorphic if
and only if 
$f_{i,j} = f_{i,j}'$ for $1\leq i \leq r$ and $1\leq j \leq t$.
We conclude the above discussion in the following result: 

\begin{thm}\label{710}
Suppose an embedding $\iota:L\hookrightarrow D$ is given, 
where $L$ is an unramified extension over $F$.

\text{\rm (1)} The set $\text{\bf FL}(O_{L},O_{\Delta}, m, \iota,
\vec{f})$ of isomorphism classes 
of $(O_{L},O_{\Delta})$-flags in $V$ of type $(\iota,\vec{f} =
(f_{1},...,f_{r}))$ can be parametrized by the vectors 
$\vec{f}_*=\big(f_{i,j}\big)_{1\leq i \leq r, 1 \leq j \leq t}$
of non-negative integers $f_{i,j}$ satisfying
\begin{equation}
  \label{eq:78}
  \frac{[L:F]}{t} \cdot
\sum_{j=1}^{t} f_{i,j} = f_{i} \quad 
\text{ for all } 1\leq i \leq r.
\end{equation}

\text{\rm (2)} Let $N_{*}$ be an $(O_{L},O_{\Delta})$-flag in $V$ 
of type $\vec{f}$. 
Let $\vec{f}_* = \big(f_{i,j}\big)_{1\leq i \leq r,1\leq j \leq t}$ 
be its corresponding vector. 
Then
$$\text{\rm End}_{(O_{L},O_{\Delta})}(N_{*}) \cong \text{\rm
  Mat}_{m'}(\vec{f}_*^o,O_{\Delta'}) \subset \text{\rm
  Mat}_{m'}(\Delta')$$ 
where the vector $\vec{f}_*^o$ is obtained by removing 
zero entries of $\vec{f}_*$. \\
\end{thm}


\begin{thank} 
  The authors would like to thank E.-U.~Gekeler and 
  Jing Yu for their steady interest and
  encouragements. The paper was revised while the second named author
  visited the IEM, Universit\"at Duisburg-Essen. He wishes to thank
  the institution for kind hospitality and good working conditions. 
  The first named author was partially supported by 
  the grant NSC 100-2811-M-007-046.
  The second named author was partially supported by the
  grants NSC 100-2628-M-001-006-MY4 and AS-99-CDA-M01. 
\end{thank}













\end{document}